\documentclass[11pt,reqno]{amsart}
\usepackage[top=1.2in, bottom=1.2in, left=1.5in, right=1.5in]{geometry}

\usepackage{amsfonts, amssymb, amsmath, enumerate, url}

\numberwithin{equation}{section}

\DeclareMathOperator{\E}{\mathbb{E}}
\DeclareMathOperator{\Var}{Var}

\DeclareMathOperator*{\sign}{sign}
\DeclareMathOperator*{\diam}{diam}
\DeclareMathOperator*{\conv}{conv}

\renewcommand{\Pr}[2][]{\mathbb{P}_{#1} \left\{ #2 \rule{0mm}{3mm}\right\}}
\newcommand{\ip}[2]{\left\langle#1,#2\right\rangle}
\newcommand{\twonorm}[1]{\left\|#1\right\|_2}
\newcommand{\abs}[1]{|#1|}
\newcommand*{\del}{\partial}

\def \R {\mathbb{R}}

\def \XX {\mathcal{X}}

\def \e {\varepsilon}
\def \d {\delta}

\def \s {\sigma}

\def \tran {\mathsf{T}}

\def \< {\langle}
\def \> {\rangle}

\def \one {{\bf 1}}

\def \xlin {\widehat{x}_{\mathrm{lin}}}
\def \xhat {\widehat{x}}

\newtheorem{theorem}{Theorem}[section]
\newtheorem{proposition}[theorem]{Proposition}
\newtheorem{corollary}[theorem]{Corollary}
\newtheorem{lemma}[theorem]{Lemma}
\newtheorem{definition}[theorem]{Definition}

\theoremstyle{remark}
\newtheorem{remark}[theorem]{Remark}

\newtheorem*{claim}{Claim}

\begin{document}
\title{High-dimensional estimation with geometric constraints}
\thanks{Y.P. was partially supported by an NSF Postdoctoral Research Fellowship under award No. 1103909. R.V. was partially supported by NSF grants DMS 1161372, 1001829, 1265782 and USAF Grant FA9550-14-1-0009. E.Y. was partially supported by an NSF Postdoctoral Research Fellowship under award No. 1204311.}

\author{Yaniv Plan \and Roman Vershynin}
\address{Y. Plan and R. Vershynin, University of Michigan Department of Mathematics, 2074 East Hall, 530 Church Street, Ann Arbor, MI 48109}

\author{Elena Yudovina}
\address{E. Yudovina, University of Michigan Department of Statistics, 439 West Hall, 1085 S. University Ave., Ann Arbor, MI 48109}

\email{\{plan, romanv, yudovina\}@umich.edu}

\date{April 14, 2014}


\begin{abstract}
Consider measuring a vector $x \in \R^n$ through the inner product with several measurement vectors, $a_1, a_2, \hdots, a_m$.  It is common in both signal processing and statistics to assume the linear response model $y_i = \ip{a_i}{x} + \e_i$, where $\e_i$ is a noise term.  However, in practice the precise relationship between the signal $x$ and the observations $y_i$ may not follow the linear model, and in some cases it may not even be known.  To address this challenge, in this paper we propose a general model where it is only assumed that {\em each observation $y_i$ may depend on $a_i$ only through $\ip{a_i}{x}$.} We do not assume that the dependence is known. This is a form of the semiparametric-single index model, and it includes the linear model as well as many forms of the \textit{generalized linear model} as special cases.  
We further assume that the signal $x$ has some structure, and we formulate this as a general assumption that $x$ belongs to some 
known (but arbitrary) feasible set $K \subseteq \R^n$. 
We carefully detail the benefit of using the signal structure to improve estimation.  The theory is based on the \textit{mean width} of $K$, a geometric parameter which can be used to understand its effective dimension in estimation problems.  We determine a simple, efficient two-step procedure for estimating the signal based on this model---a linear estimation followed by metric projection onto $K$.  
We give general conditions under which the estimator is minimax optimal up to a constant.  
This leads to the intriguing conclusion that in the high noise regime, an unknown non-linearity in the observations does not significantly reduce one's ability to determine the signal, even when the non-linearity may be non-invertible.  Our results may be specialized to understand the effect of non-linearities in compressed sensing.
\end{abstract}

\maketitle

\section{Introduction}
\label{sec: intro}

An important challenge in the analysis of high-dimensional data is to combine noisy observations---which individually give uncertain information---to give a precise estimate of a signal.  One key to this endeavor is to utilize some form of structure of the signal.  If this structure comes in a low-dimensional form, and the structure is aptly used, the resulting dimension reduction can find a needle of signal in a haystack of noise.  This idea is behind many of the modern results in \textit{compressed sensing} \cite{CSbook}.

To make this more concrete, suppose one is given a series of measurement vectors $a_i$, paired with observations $y_i$.  It is common to assume a linear data model $y_i = \ip{a_i}{x} + \e_i$, relating the response to an unknown signal, or parameter vector, $x$.  
However, in many real problems the linear model may not be justifiable, or even plausible---consider binary observations.  Such data can be approached with the \textit{semiparametric single index model}, in which one models the data as 
$$
y_i = f(\ip{a_i}{x} + \e_i)
$$ 
for some unknown function $f: \R \rightarrow \R$.  The goal is to estimate the signal $x$ despite the unknown non-linearity $f$.  In fact, one may take a significantly more general approach as described in the next section.


It is often the case that some prior information is available on the structure of the signal $x$. Structure may come in many forms; 
special cases of interest include manifold structure, sparsity, low-rank matrices and compressible images. 
We consider a general model, assuming only that 
$$
\text{$x \in K$ for some closed star-shaped%
\footnotemark
set $K \subset \R^n$.}
$$
\footnotetext{A set $K$ is called star-shaped if $\lambda K \subseteq K$ whenever $0 \leq \lambda \leq 1$.}
This includes cones and convex sets containing the origin.  We focus on three goals in this paper:  
\begin{enumerate}[\quad 1)]
  \item determine general conditions under which the signal can be well estimated;
  \item give an efficient method of estimation; 
  \item determine precisely what is gained by utilizing the feasible set $K$ that encodes signal structure.  
\end{enumerate}



Let us outline the structure of this paper. 
In the rest of this section, we carefully specify our model, describe the proposed estimator, and analyze its performance for general $K$; the main result is Theorem~\ref{thm: main} in Section~\ref{ssec: main result}.  
In Section~\ref{sec: examples of K}, we specialize our results to a number of standard feasible sets $K$, including sparse vectors and low-rank matrices.  In Section~\ref{sec: examples of y}, we specialize our results to specific versions of the semiparametric single-index model.  We include the linear model, the binary regression model, and a model with explicit non-linearity.  In Section~\ref{sec: optimality}, we discuss the optimality of the estimator.  We show that it is minimax optimal up to a constant under fairly general conditions.  In Section~\ref{sec: comparison to M^*} we illustrate the connection between our results and a deep classical estimate from \textit{geometric functional analysis}, called the \textit{low $M^*$ estimate}.  This estimate tightly controls the diameter of a random section of an arbitrary set $K$.  Section~\ref{sec: elena's section} gives an overview of the literature on the semiparametric single-index model, emphasizing the differences between the more geometric approach in this paper and the classical approaches used in statistics and econometrics.  We give some concluding remarks in Section~\ref{sec: discussion}. Sections~\ref{s: proofs}--\ref{sec: optimality proofs} contain the technical proofs. In Section~\ref{s: proofs} we give the proofs of the results from Section~\ref{sec: intro}, including our main result; in Section~\ref{s: thm: main whp} we state and prove a version of our main result which holds with high probability rather than in expectation; and in Section~\ref{sec: optimality proofs} we give proofs of the optimality results from Section~\ref{sec: optimality}.

\medskip
Throughout this paper, we use notation common in the compressed sensing literature. However, because we expect our results to be of interest to a wider statistical audience, we provide a dictionary of notation in Section~\ref{sec: elena's section}, Table~\ref{table}.

\subsection{Model}
\label{ssec: model}
Let $x \in \R^n$ be a fixed (unknown) \emph{signal} vector, and let $a_i \in \R^n$ be independent random measurement vectors. 
We assume that observations $y_i$ are independent real-valued random variables such that 
\begin{equation} \label{single-index model} 
\text{\em each observation $y_i$ may depend on $a_i$ only through $\ip{a_i}{x}$}.
\end{equation} 
In other words, we postulate that, given $\ip{a_i}{x}$, the observation $y_i$ and the measurement vector $a_i$ are conditionally independent. We are interested in recovering the signal vector $x$ using as few observations as possible. 
Furthermore, we will usually assume some \emph{a priori} knowledge about $x$ of the form $x \in K$ for some known set $K \subset \R^n$.  Note that the norm of $x$ is sacrificed in this model since it may be absorbed into the dependence of $y_i$ on $\ip{a_i}{x}$.  Thus, it is of interest to estimate $x$ up to a scaling factor.

Unless otherwise specified, we assume that $a_i$ are independent standard normal vectors in $\R^n$. We lift this assumption in several places in the paper: see the matrix completion problem in Section~\ref{ssec: low-rank}, and the lower bounds on all possible estimators in Section~\ref{sec: optimality}.  We make further suggestions of how this assumption may be generalized in Section~\ref{sec: discussion}.


\subsection{Linear estimation}
\label{ssec: linear estimation}

The first and simplest approach to estimation is to ignore the feasible set $K$ for a moment.  One may then employ the following {\em linear estimator}
$$
\xlin := \frac{1}{m} \sum_{i=1}^m y_i a_i.
$$
It is not difficult to see that $\xlin$ is an unbiased estimator of the properly scaled vector $x$,
and to compute the mean squared error. This is the content of the following proposition
whose proof we defer to Section~\ref{s: proofs}.

\begin{proposition}[Linear estimation]			\label{prop: linear estimation}
  Let $\bar{x} = x/\|x\|_2$. Then
  $$
  \E \xlin = \mu \bar{x} 
  \quad \text{and} \quad 
  \E \|\xlin - \mu \bar{x}\|_2^2 = \frac{1}{m} \Big[ \s^2 + \eta^2(n-1) \Big].
  $$
  Here
  \begin{equation}         \label{eq: mu sigma eta}
  \mu = \E y_1 \ip{a_1}{\bar{x}}, \quad \s^2 = \Var(y_1 \ip{a_1}{\bar{x}}), \quad \eta^2 = \E y_1^2.
  \end{equation}
\end{proposition}

By rotation invariance of $a_i$, the parameters $\mu$, $\sigma$ and $\eta$ depend on the magnitude $\|x\|_2$ 
but not on the direction $\bar{x}=x/\|x\|_2$ of the unknown vector $x$ or on the number of observations $m$. 
These properties make it simple to compute or bound these parameters in many important cases, 
as will be clear from several examples below. For now, it is useful to think of these parameters as constants.

\medskip

The second part of Proposition~\ref{prop: linear estimation} essentially states that 
\begin{equation}         \label{eq: linear estimation asymp}
\Big[\E \|\xlin - \mu \bar{x}\|_2^2 \Big]^{1/2} \asymp \frac{1}{\sqrt{m}} \big[ \s + \eta \sqrt{n} \big].
\end{equation}
We can express this informally as follows:
\begin{equation}   \label{eq: linear essence}
  \text{{\em Linear estimation is accurate for $m = O(n)$ observations.}}
\end{equation}

\subsection{Projection onto the feasible set}

Although estimate \eqref{eq: linear estimation asymp} on the accuracy of linear estimation is sharp, 
it can be significantly improved if some {\em prior information} 
is available about the vector $x$. A rigorous way to encode the prior information 
would be to assume that $x \in K$ where $K \subset \R^n$ is some fixed, closed, and known 
{\em feasible set}.  
Because the scaling of $x$ may be absorbed into the semiparametric single-index model, it is actually more natural to assume that 
\begin{equation}         \label{eq: K}
\mu \bar{x} \in K
\end{equation}
in the notation of Proposition~\ref{prop: linear estimation}.\footnote{Passing between assumptions
$x \in K$ and $\mu \bar{x} \in K$ in practice should be painless by rescaling $K$, 
as the scaling factor $\mu$ is usually easy to estimate.  Further, if $K$ is a cone, then there is no difference between the assumptions.}

Recall that $\xlin$ is an unbiased estimator of $\mu \bar{x}$, a vector that lies in $K$.
So to incorporate the feasible set $K$ into estimation, a natural step 
is to {\em metrically project}
the linear estimator $\xlin$ onto $K$. In other words, we define
\begin{equation}         \label{eq: xhat}
\xhat = P_K(\xlin) = \arg\min_{z \in K} \|\xlin - z\|_2.
\end{equation}
As we will see shortly, this non-linear estimator outperforms the linear one, often by a 
big margin.

\subsection{Measuring the size of the feasible set by mean width}		\label{s: mean width}

The quality of the non-linear estimator \eqref{eq: xhat} 
should depend on the size of the feasible set $K$.
It turns out that there is a simple geometric notion that captures the size of $K$
for this purpose. This notion is the {\em local mean width}.
At first reading, one may replace it with a slightly simpler concept of {\em global} mean width.

\begin{definition}[Mean width]						\label{def: mean width}
  The (global, Gaussian) {\em mean width} of a subset $K \subset \R^n$ 
  is defined as 
  $$
  w(K) = \E \sup_{x,y \in K} \ip{g}{x-y},
  $$
  where $g \sim N(0, I_n)$. 
  The {\em local mean width} of a subset $K \subset \R^n$ is a function of scale $t \ge 0$,
  and is defined as
  $$
  w_t(K) = \E \sup_{x,y \in K, \, \|x-y\|_2 \le t} \ip{g}{x-y}.
  $$
\end{definition}

Note that $w_t(K) \le w(K)$ trivially holds for all $t$.
The concepts of mean width, both global and local, originate in geometric functional analysis 
and asymptotic convex geometry, see e.g. \cite{GM}.
Quantities equivalent to mean width appear also in stochastic processes 
under the name of $\gamma_2$ functional (see \cite{Talagrand}),
and in statistical learning theory under the name of Gaussian complexity (see \cite{BM}). 

More recently, the role of mean width (both local and global) was recognized in the area of signal recovery
\cite{MPT,PV_IEEE,PV_DCG,ALMT,ALPV}. To interpret these developments as well as this paper,
it is often helpful to think of the square of the mean width of the properly scaled feasible set $K$
as an {\em essential dimension} (as opposed to the algebraic dimension, which typically equals $n$).  A notable benefit of this method of measuring dimension is that it is robust to perturbations:  If $K$ is slightly increased, the mean width only changes slightly.

\subsection{Main result}
\label{ssec: main result}

\begin{theorem}[Non-linear estimation]				\label{thm: main}
  Let $\bar{x} = x/\|x\|_2$.
  Assume that $\mu \bar{x} \in K$ where $K$ is a fixed star-shaped closed subset of $\R^n$. 
  Then the non-linear estimator $\xhat$ defined in \eqref{eq: xhat} satisfies the following for every $t>0$:
  \begin{equation}         \label{eq: main}
  \E \|\xhat - \mu \bar{x}\|_2 \le t + \frac{2}{\sqrt{m}} \Big[ \s + \eta \, \frac{w_t(K)}{t} \Big].
  \end{equation}
  Here $\mu$, $\sigma$ and $\eta$ are the numbers defined in \eqref{eq: mu sigma eta}.
\end{theorem}

\medskip
In Section~\ref{sec: optimality}, we will give general conditions under which the error achieved in this theorem is minimax optimal up to a constant factor---thus the theorem gives a precise, non-asymptotic, characterization of the benefit of using the signal structure $K$ to improve the estimation. 
In Section~\ref{s: thm: main whp} we will state a
version of Theorem~\ref{thm: main} whose conclusion is valid with high probability 
rather than in expectation.
In the next two sections we will simplify Theorem~\ref{thm: main} in special cases,
show how it is superior to Proposition~\ref{prop: linear estimation}, and illustrate it
with a number of examples.

\section{Feasible sets $K$: consequences and examples}
\label{sec: examples of K}

In this section we state a simpler version of Theorem~\ref{thm: main}, 
compare Theorem~\ref{thm: main} with Proposition~\ref{prop: linear estimation}, 
and illustrate it with
several classes of feasible sets $K$ that may be of interest in applications. 

\subsection{General sets: error bounds via global mean width}		\label{s: via global mean width}

Let us state a simpler conclusion of Theorem~\ref{thm: main}, in terms of global mean width $w(K)$
and without the parameter $t$.

Let $K$ be an arbitrary compact star-shaped subset of $\R^n$. 
Replacing in \eqref{eq: main} the local mean width $w_t(K)$ by the bigger quantity $w(K)$
and optimizing in $t$, we obtain
\begin{equation} \label{eq: global mean width bound}
\E \|\xhat - \mu \bar{x}\|_2 \le \frac{2 \sigma}{\sqrt{m}} + 2 \sqrt{2} \, \Big[ \frac{\eta \, w(K)}{\sqrt{m}} \Big]^{1/2}.
\end{equation}
While this conclusion is less precise than \eqref{eq: main}, it may be sufficient in some applications. 
We note the unusual rate $O(m^{-1/4})$ in the right hand side, which nevertheless can be sharp for some signal structures such as the $\ell_1$-ball.  See Sections \ref{ssec: l1}, \ref{ssec: reverse sudakov} and also \cite{RWY_Minimax}.

\subsection{General sets: comparison with linear estimation}

Let us compare the qualities of the linear and non-linear estimators.
Let $K$ be an arbitrary star-shaped closed subset of $\R^n$.
The definition of local mean width implies that
\begin{equation}         \label{eq: w by sqrt n}
w_t(K) \le w(t B_2^n) = t \, w(B_2^n) = t \E \|g\|_2 \le t (\E \|g\|_2^2)^{1/2} = t \sqrt{n}. 
\end{equation}
(Here and below, $B_2^n$ is the $n$-dimensional Euclidean ball.)
Substituting this bound into \eqref{eq: main} and letting $t \to 0$, we deduce 
the following estimate from Theorem~\ref{thm: main}:
$$
\E \|\xhat - \mu \bar{x}\|_2 \le \frac{2}{\sqrt{m}} \big[ \s + \eta \sqrt{n} \big].
$$
Note that this is the same upper bound as 
Proposition~\ref{prop: linear estimation} gives for linear estimation, up to an absolute constant factor.
We can express this conclusion informally as follows:
\begin{quote}
  {\em Projecting the linear estimator $\xlin$ onto the feasible set $K$ 
  can only improve the accuracy of estimation.}
\end{quote}
We will shortly see that such improvement is often significant.

\subsection{General cones}

Let $K$ be a fixed closed cone in $\R^n$, so $tK = K$ is satisfied for all $t \ge 0$. 
Then 
$$
w_t(K) = t \, w_1(K).
$$
Substituting this into \eqref{eq: main} and letting $t \to 0$, 
we can state Theorem~\ref{thm: main} in this case as follows.

\begin{theorem}[Estimation in a cone]				\label{thm: cone}
  Assume that $x \in K$ where $K$ is a fixed closed cone in $\R^n$.
  Let $\bar{x} = x/\|x\|_2$.
  Then
  \begin{equation}         \label{eq: cone}
  \E \|\xhat - \mu \bar{x}\|_2 \le \frac{2}{\sqrt{m}} \big[ \sigma + \eta \, w_1(K) \big].
  \end{equation}
  Here $\mu$, $\sigma$ and $\eta$ are the numbers defined in \eqref{eq: mu sigma eta}.
\end{theorem}

To further simplify the bound \eqref{eq: cone}, note that term containing $w_1(K)$ 
essentially dominates there. This is based on the following observation.  

\begin{lemma}					\label{lem: w below}
  For a nonempty cone $K$, one has 
  $\E w_1(K) \ge \sqrt{2/\pi}$.
\end{lemma}

\begin{proof}
Note that $K-K$ contains a line, so $(K-K) \cap B_2^n$ contains a pair 
of antipodal points on the unit sphere. Therefore $w_1(K) = w( (K-K) \cap B_2^n )$ 
is bounded below by the first absolute moment of the standard normal distribution, 
which equals $\sqrt{2/\pi}$.
\end{proof}

Using Lemma~\ref{lem: w below}, we see that \eqref{eq: cone}
  implies that 
  \begin{equation}         \label{eq: cone simplified}
  \E \|\xhat - \mu \bar{x}\|_2 \le \gamma \, \frac{w_1(K)}{\sqrt{m}}, 
  \quad \text{where} \quad 
  \gamma = \sqrt{2\pi} \, \s + 2\eta. 
  \end{equation}
  
\medskip

In particular, we arrive at the following informal conclusion:
\begin{quote}
  {\em Linear estimation followed by projection onto the feasible cone $K$ 
  is accurate for $m = O(w_1(K)^2)$ observations.}
\end{quote} 

\begin{remark}[Essential dimension]
  This result becomes especially transparent if we think of $w_1(K)^2$ as 
  the {\em essential dimension} of the cone $K$. 
  A good estimation is then guaranteed for {\em the number of observations proportional 
  to the essential dimension} of $K$. 
  This paradigm manifested itself in a number of recent results in the area of signal recovery 
  \cite{MPT,PV_IEEE,PV_DCG,ALMT,ALPV}. 
  Among these we especially note \cite{ALMT} where the results are explicitly stated in terms of 
  a ``statistical dimension'', a very close relative of $w_1(K)$. 
\end{remark}

\begin{remark}[Projecting onto $S^{n-1}$]
Since magnitude information about $x$ is irrecoverable in the semiparametric model, and $K$ has a conic, scale-invariant structure in this subsection, it is natural to rephrase our theorem in terms of the estimation of $\bar{x}$, unscaled.  One may then normalize $\xhat$ to come to the following conclusion:
\begin{equation}
\twonorm{\frac{\xhat}{\twonorm{\xhat}} - \bar{x}} \leq \frac{2 \gamma}{\mu} \cdot \, \frac{w_1(K)}{\sqrt{m}}.
\end{equation}
Indeed, this follows because
\[\twonorm{\frac{\xhat}{\twonorm{\xhat}} - \bar{x}} \leq \twonorm{\frac{\xhat}{\twonorm{\xhat}} - \frac{\xhat}{\mu}} + \twonorm{\frac{\xhat}{\mu} - \bar{x}} \leq 2 \twonorm{\frac{\xhat}{\mu} - \bar{x}}\]
where the last step follows since $\xhat/\twonorm{\xhat}$ is the closest point in $S^{n-1}$ to $\xhat/\mu$.
\end{remark}

A quick computation similar to \eqref{eq: w by sqrt n} shows that the essential dimension of a cone is always bounded by its algebraic dimension. Thus the non-linear estimator $\xhat$ outperforms the linear estimator $\xlin$ discussed in \eqref{eq: linear essence}, and the improvement is dramatic in cases where $w_1(K)^2 \ll n$. We give examples of such situations below.

\subsection{The set of sparse vectors}
\label{ssec: sparse}

Sparsity is a key signal structure used in many modern applications.  For example, it is common to assume that only a small subset of coefficients are significant in a regression model.  As another example, images are generally compressible in some dictionary, i.e., up to a linear transformation an image is sparse.  The \textit{compressed sensing} model (see the book \cite{CSbook}) is based on sparsity, and one may use the results in this paper as a treatment of unknown non-linearities in compressed sensing.

The sparsity model takes $K = \{x \in \R^n: |\text{supp}(x)| \leq s\}$, i.e., $K$ is the set of vectors with at most $s$ nonzero entries.  Fortunately, projecting onto $K$ is computationally efficient: one only needs to retain the $s$ largest entries of the vector, replacing all other entries by 0.  This is referred to as \textit{hard thresholding}.  Thus, the projection estimator is quite amenable to large data.  

Further, when $s$ is significantly smaller than $n$, projecting onto $K$ gives a fruitful dimension reduction.  This is evident from the calculation
\[c \sqrt{s \log (2n/s)} \le w_1(K) \le C \sqrt{s \log (2n/s)}\]
which is given in \cite[Lemma 2.3]{PV_IEEE}.  Thus, $m \sim s \log (2n/s)$ observations are sufficient to estimate 
an $s$-sparse vector in $\R^n$.

{\bf Sparse in a dictionary.} In many applications, the signal may be sparse in a dictionary, rather than canonically sparse, e.g., images are often compressible in a wavelet basis.  In this case, the feasible set $K \subseteq \R^q$ is defined as 
$$
K = \{Dv: \abs{\text{supp}(v)} \leq s\}
$$ 
where the matrix $D \in \R^{q \times n}$ is referred to as a dictionary.  
A computation of the essential dimension of $K$ can be made 
through a slight generalization of \cite[Lemma 2.3]{PV_IEEE} (via a covering argument), which shows that once again
$$
w_1(K) \leq C \sqrt{s \log(2n/s)}.
$$
While projection onto $K$ can pose a challenge, we will consider signals which are approximately sparse in a dictionary in Section \ref{ssec: l1}; in this case projection can be done efficiently through convex programming.

\subsection{The set of low-rank matrices}
\label{ssec: low-rank}

The low-rank signal structure is prevalent in statistical applications, thus leading to the popularity of \textit{principal component analysis}; for many examples, see \cite{KV}.  Reconstructing a low-rank matrix from linear combinations of entries arises in various applications such as quantum state tomography or recommender systems \cite{recht2010guaranteed,gross2010quantum}.
Recently, there has been quite a bit of interest and new theory on the problem of \textit{matrix completion}: reconstruction of a low-rank matrix from a subset of its entries \cite{MC,tight_low_rank}, with binary non-linearities given special consideration \cite{1-bit_MC,Srebro}.  In this subsection, we first specialize our theory to the low-rank structure and then adjust our theory to the matrix completion model.

Here we take the cone $K$ to be the set of $d_1 \times d_2$ matrices with rank bounded by some small $r$.  Projection onto $K$ requires only taking the singular value decomposition and keeping the largest $r$ singular values and singular vectors. This procedure is referred to as {\em hard thresholding of singular values}. The mean width of $K$ is easy to estimate (see e.g.  \cite{PV_IEEE}); one has
$$
w_1(K) \leq \sqrt{2r (d_1 + d_2)}.
$$  
Thus, $m \sim r (d_1 + d_2)$ observations are sufficient to estimate a $d_1 \times d_2$ matrix of rank $r$.

\medskip

We now move away from the assumption of Gaussian measurement vectors in order to accommodate the matrix completion model.  In this case, one observes a random sample of the entries of the matrix $X$ (linear observations); that is, the measurement ``vectors'' (matrices) $a_i$ are uniformly distributed on the set of matrices with exactly one nonzero entry, and we take $y_i = \ip{a_i}{X}$.  Projection-based estimators have been considered for matrix completion in \cite{Montanari_MC,Kholtchinskii_MC,Chatterjee_MC}.  These papers take general models, which require a somewhat more complicated theory.  We consider a simple projection-based estimator and prove that it is quite accurate under a simple model, recovering best-known theory for this model. The proof takes much in common with known matrix completion theory; the purpose of putting it here is to emphasize its shortness under a simple model.  Further, this gives an intuition as to how to extend our results in general to non-Gaussian measurements.

\smallskip

Let us consider the following simple {\em matrix completion model.}  Take $d_1 = d_2 = d$ 
and consider a random subset $\Omega \subset \{1, \hdots, d\} \times \{1,\hdots, d\}$, 
which includes each entry $(i,j)$ with probability $p$, independently of all other entries. 
Let the observations $y_i$ give the entries of the matrix $X$ contained in $\Omega$. 
Further, assume that $X$ satisfies the following \textit{incoherence condition}:  
each entry of $X$ is bounded by a parameter $\zeta$.  

\smallskip

Let us specify the {\em estimator} of $X$. Define the mask matrix $\Delta_{\Omega}$ with $0,1$ entries by
$$
(\Delta_{\Omega})_{ij} = \one_{\{(i,j) \in \Omega\}}.
$$
Then the linear estimator discussed in Section~\ref{ssec: linear estimation}, $\frac{1}{m} \sum_{i=1}^m y_i a_i$, is naturally replaced with the Hadamard (entry-wise) product $\frac{1}{p} \Delta_{\Omega} \circ X$.  Clearly, this is a rough estimator of $X$, but nevertheless it is an unbiased estimator.  Fortunately, the projection step can significantly improve the rough estimator.  As above, project $\frac{1}{p} \Delta_{\Omega} \circ X$ onto the set of rank-$r$ matrices, and call this $\widehat{X}$.  

Under this model, we prove that the estimator is accurate, thereby recovering best-known theory for this model.

\begin{proposition}[Matrix completion accuracy]  \label{prop: MC}  Consider the model described above.  Let $m := \E |\Omega| = p d^2$, and assume that $m \geq d \log d$.  The estimator $\widehat{X}$ has the following average error
per entry:
$$
\frac{1}{d} \E \|\widehat{X} - X\|_F \leq C \sqrt{\frac{rd}{m}} \, \zeta.
$$
\end{proposition}

\begin{remark}[The benefit of projection]
Despite the non-Gaussian form of the measurements, the estimation error in matrix completion is proportional to $\sqrt{w_1(K)/m} = \sqrt{rd/m}$.
\end{remark}

\begin{remark}
As is evident from the proof below, if the observations are corrupted with i.i.d. $N(0, \nu^2)$ noise, the error bound becomes
$$
\frac{1}{d} \E \|\widehat{X} - X\|_F \leq C \sqrt{\frac{rd}{m}} \, (\zeta + \nu).
$$
When $\nu \geq \zeta$, this error is known to be minimax optimal up to a constant.  

The noise is easily incorporated and carried through the proof.  At the end it is simply necessary to replace the scaled Rademacher matrix $\zeta\cdot R$ by the sum of $\zeta\cdot R$ and a noise matrix.
\end{remark}  

\begin{proof}[Proof of Proposition \ref{prop: MC}.]  
 Note that rank($\widehat{X} - X) \leq 2r$. It follows that
\[\|\widehat{X} - X\|_F \leq \sqrt{2 r} \|\widehat{X} - X\|.\]
The operator norm may be bounded as follows:
\[\|\widehat{X} - X\| \leq \|\widehat{X} - p^{-1} \Delta_{\Omega} \circ X\| + \| p^{-1} \Delta_{\Omega} \circ X - X\| \leq 2 p^{-1} \|\Delta_{\Omega} \circ X - p X\|,\]
where the last step follows since $\widehat{X}$ is the closest rank-$r$ matrix to $p^{-1} \Delta_{\Omega} \circ X$ in operator norm.  The right-hand side is the operator norm of a matrix with independent, mean-zero entries and may be controlled directly in many ways.  To give an optimal result without log factors, one may use techniques from empirical process theory and random matrix theory.  By \textit{symmetrization} \cite[Lemma 6.3]{LT} followed by the \textit{contraction principle} \cite[Theorem 4.4]{LT},
\[
\E \|\Delta_{\Omega} \circ X - p X\| \leq 2 \E \| R \circ \Delta_{\Omega} \circ X\| \leq 2 \zeta \E \|R \circ \Delta_{\Omega}\|
\]
where $R$ is a matrix whose entries are independent Rademacher random variables (i.e. $R_{ij} = 1$ or $-1$ with probability $1/2$).  Thus, $R \circ \Delta_\Omega$ has independent, identically distributed entries.  Seginer's theorem \cite{seginer} then gives
\[
\E \|R \circ \Delta_{\Omega} \| \leq C \sqrt{p d}
\]
where $C$ is a numerical constant.  For the above inequality to hold, we must have $p \geq \log(d)/d$, i.e., on average at least $\log(d)$ observations per row.  Putting all of this together gives
\[
\E\|\widehat{X} - X\|_F \leq C \sqrt{\frac{r d}{p}} \zeta.
\]
Divide through by $d$ and recall $m = p d^2$ to conclude.
\end{proof}

\subsection{The set of approximately sparse vectors: $\ell_1$-ball }
\label{ssec: l1}

In real applications, exact sparsity is unusual, and is usually replaced with approximate sparsity. A clean way to do this is to assume that the signal belongs to a scaled $\ell_p$ ball in $\R^n$ with $0 < p \leq 1$.  The $\ell_1$-ball, denoted $B_1^n$, 
is an especially useful signal structure due to its convexity, see, e.g., \cite{PV_IEEE}. 

To see the connection between sparsity and the $\ell_1$-ball, consider an $s$-sparse vector $x$, with Euclidean norm bounded by 1.  Then, by Cauchy-Schwartz inequality, $\|x\|_1 \leq \sqrt{s} \twonorm{x} \leq \sqrt{s}$ and thus $x \in \sqrt{s} B_1^n$.  However, one may take a slight perturbation of $x$, without significantly changing the $\ell_1$ norm, and thus the set accommodates approximately sparse vectors.  

Let $K = \sqrt{s} B_1^n$. For simplicity, we compute the global (as opposed to local) mean width, which is 
\[w(K) = \E \sup_{x \in 2\sqrt{s} B_1^n} \< x , g \> = 2 \sqrt{s} \E \|g\|_{\infty} \leq 4 \sqrt{2 s \log n}.\]
The last inequality comes from the well-known fact $\E \|g\|_{\infty} \leq 2 \sqrt{2 \log n}$.  Plugging into \eqref{eq: global mean width bound} gives
\[\E \|\xhat - \mu \bar{x}\|_2 \le \frac{2 \sigma}{\sqrt{m}} + 16 \, \Big[ \frac{\eta \, \sqrt{s \log n}}{\sqrt{m}} \Big]^{1/2}.\]
Thus, $m \sim s \log n$ observations are sufficient to estimate an approximately $s$-sparse vector in $\R^n$.

In Section~\ref{ssec: reverse sudakov} we give a careful treatment of the $\ell_1$-ball using local mean width. 
We will do a (well known) calculation showing that this improves the above error bound slightly (roughly, $\log n$ can be replaced by $\log(2n/s)$). We will find that the true error rate is minimax optimal for many models of the functional dependence of $y_i$ on $\ip{a_i}{x}$ when $m$ is not too large ($m \leq C \sqrt{n}$).  We emphasize the unusual dependence $m^{-1/4}$ in this non-asymptotic regime.

{\bf Approximately sparse in a dictionary.}  Consider a dictionary $D$ (see Section \ref{ssec: sparse}).  By replacing $B_1^n$, with $D B_1^n$, one has a set encoding approximate sparsity in the dictionary $D$.  Fortunately, due to its convexity, projection onto this set may be performed efficiently.  Further, it is straightforward to bound the mean width.  Indeed, Slepian's inequality \cite{LT} gives
\[w(D B_1^n) \leq \|D\| \cdot w(B_1^n) \leq 4 \|D\| \sqrt{2 s \log n}.\]
In other words, if the operator norm of $D$ is bounded, then the set $D B_1^n$ has similar essential dimension---and therefor error bound---as the set $B_1^n$. Summarizing, we find that $m \sim s \log n$ observations are sufficient to estimate a 
vector that is $s$-sparse in a dictionary of $n$ elements.

\section{Observations $y_i$: consequences and examples}
\label{sec: examples of y}

We specialize our results to some standard observation models in this section.


\subsection{Linear observations}				\label{s: linear observations}

The simplest example of $y_i$ are linear observations
\begin{equation}         \label{eq: linear observations}
y_i = \ip{a_i}{x}.
\end{equation}
The parameters in \eqref{eq: mu sigma eta} are then 
$$
\mu = \|x\|_2, \quad \sigma = \sqrt{2}\, \|x\|_2, \quad \eta = \|x\|_2. 
$$
Substituting them into \eqref{eq: main} and \eqref{eq: cone simplified}, we obtain the following result.

\begin{corollary}[Estimation from linear observations]
  Assume that $x \in K$ where $K$ is a star-shaped set in $\R^n$.
  Assume the observations $y_i$ are given by \eqref{eq: linear observations}. 
  Then for every $t > 0$,
  \[
  \E \|\xhat - x\|_2 \le t + \frac{2 \twonorm{x}}{\sqrt{m}}\Big[\sqrt{2} + \eta \frac{w_t(K)}{t}\Big].
 \]
  If $K$ is a cone, then 
  $$  
  \E \|\xhat - x\|_2 \le C \, \frac{w_1(K)}{\sqrt{m}} \, \|x\|_2
  $$
  where $C = 2(\sqrt{\pi}+1) \approx 5.54$.   
\end{corollary}

Note that here we estimate the signal $x$ itself, rather than its scaled version as in our previous results.

\subsection{Noisy linear observations}
\label{ssec: noisy linear}

A more general class of examples includes noisy linear observations of the form
\begin{equation}         \label{eq: noisy linear observations}
y_i = \ip{a_i}{x} + \e_i
\end{equation}
where  $\e_i$ are mean-zero, variance $\nu^2$, random variables which are independent of each other and of $a_i$.
Clearly, such observations follow the single-index model \eqref{single-index model}.

A straightforward computation of the parameters in \eqref{eq: mu sigma eta} shows that
$$
\mu = \|x\|_2, \quad \sigma = \sqrt{2 \|x\|_2^2 +\nu^2} \le \sqrt{2} \twonorm{x} + \nu, \quad \eta = \sqrt{\|x\|_2^2+\nu^2} \leq \twonorm{x} + \nu. 
$$
Substituting them into \eqref{eq: main} and \eqref{eq: cone simplified}, we obtain the following result.

\begin{corollary}[Estimation from noisy linear observations]
\label{cor: noisy linear}
  Assume that $x \in K$ where $K$ is a star-shaped subset of $\R^n$.
  Assume that the observations $y_i$ are given by \eqref{eq: noisy linear observations}. 
  Then
  \begin{equation}
  \label{eq: noisy linear}
  \E \|\xhat - x\|_2 \le t + C \frac{\twonorm{x} + \nu}{\sqrt{m}}\Big[1 + \frac{w_t(K)}{t}\Big]
  \end{equation}
  where $C = 2 \sqrt{2} \approx 2.83.$  If $K$ is a cone then
  $$  
  \E \|\xhat - x\|_2 \le C' (\|x\|_2+\nu) \; \frac{w_1(K)}{\sqrt{m}}
  $$
  where $C' = 2(\sqrt{\pi} + 1) \approx 5.54$.   
\end{corollary}

\begin{remark}[Signal buried in noise] One observes that both the size of the signal, $\twonorm{x}$, and the size of the noise, $\nu$, contribute to the error bound.  When the noise is larger than the signal, the estimate may still be quite accurate because of the dimension reduction gained by projecting onto $K$.  In fact, we will show that the error is minimax optimal up to a multiplicative constant under general conditions on $K$ in Section~\ref{sec: optimality}.     When the signal is larger than the noise, the proposed estimator may not be optimal, depending on $K$.  Indeed, if $K = \R^n$, $m \geq n$, and $\nu = 0$, the minimax error is 0.  However, as a general theme in this paper we concentrate on rough observations for which 0 error is impossible---noiseless linear observations being the exception to this rule.  The case of noiseless linear observations are handled well by the generalized Lasso estimator, as analyzed in the follow-up paper \cite{plan2015generalized}.  In that case, the size of the signal does not appear in the error bound.
\end{remark} 

\begin{remark}[Generalizing assumptions]
We also note that the corollary could be adapted to the case when the random variables $\e_i$ expressing the noise are not mean zero, 
and do depend on $x$ (but only through $\ip{a_i}{x}$).
\end{remark}

\subsection{Nonlinear observations}

A general class of examples satisfying the single-index model \eqref{single-index model} 
consists of nonlinear observations of the form
\begin{equation}         \label{eq: non-linear measurements}
y_i = f(\ip{a_i}{x}).
\end{equation}
Here $f : \R \to \R$ is a fixed link function, which may be unknown.  In particular, $f$ may be discontinuous and not one-to-one (like the sign function), or even nonmonotonic in its argument (like the sine function).

The rotation invariance of $a_i$ allows us to express the parameters in \eqref{eq: mu sigma eta} 
as functions of $f$ and $\|x\|_2$ only. Indeed, 
\begin{equation}         \label{eq: mu sigma eta nonlinear}
\mu = \E f(g\|x\|_2)g, \quad \s^2 = \Var \big[ f(g\|x\|_2)g \big], \quad \eta^2 = \E f(g\|x\|_2)^2,
\end{equation}
where $g \sim N(0,1)$.
Substituting this into Theorem~\ref{thm: main} or Theorem~\ref{thm: cone}, 
we can bound the error of estimation in terms of $f$, the magnitude of signal $\|x\|_2$,
and the local mean width of $K$; we are leaving this to the interested reader.

It is important and surprising that our estimation procedure---computing $\xhat$ from 
the observations $y_i$---does not depend on the function $f$. In words, 
\begin{quote}
  {\em Any knowledge of the non-linearity $f$ defining the observations $y_i$ is not needed
  to estimate $x$ from these observations.}
\end{quote}
The quality of estimation must of course depend on $f$, and the exact dependence is
encoded in the parameters $\mu$, $\sigma$ and $\eta$ in \eqref{eq: mu sigma eta nonlinear}.
However, one does not need to know $f$ exactly to compute reasonable bounds for these parameters. 

Roughly, $\sigma$ and $\eta$, which play a similar role to a noise level, are well upper bounded if $f(g)$ does not have extremely heavy tails.  On the other hand, $\mu$ which plays a role similar to signal strength, can be lower-bounded when $f$ is monotonically increasing (although this is not necessary).  A notable exception is when $f$ is an even function and $\mu = 0$.  In that case, the method of signal estimation described in this paper would be inappropriate.  In Section~\ref{ssec: upper vs lower} we give a more precise treatment bounding these parameters, and thus derive uniform error bounds for a fairly large class of functions $f$.


\subsection{Generalized linear models}

An even more general class of examples for the single-index model \eqref{single-index model} 
consists of the noisy non-linear observations
$$
y_i = f(\ip{a_i}{x} + \e_i) + \d_i.
$$
Here $\e_i$ are random variables independent of each other and of $a_i$, 
and the same is assumed about $\d_i$. (However, $\e_i$ may be correlated with $\d_i$.)  This is a variation on the \textit{generalized linear model}, which assumes
\[\E [y_i \vert a_i] = f(\ip{a_i}{x})\]
for some known function $f$.%
\footnote{The generalized linear model also usually assumes that $y_i$ belongs to an \textit{exponential family}.}
In contrast, we take the semiparametric approach and assume that the non-linearity, $f$, is unknown. (On the other hand, GLMs typically assume more general noise models than we do; we discuss the connections between single-index models and GLMs further in Section~\ref{sec: elena's section}.) To specialize Theorem \ref{thm: main} to this case, we need only control $\mu, \sigma,$ and $\eta$.

\subsection{Binary observations and logistic regression}
\label{ssec: binary}

An important class of examples is formed by binary observations---those satisfying the model \eqref{single-index model} and such that 
\begin{equation}         \label{eq: binary observations}
y_i \in \{-1,1\}^n, \quad \E y_i = f(\ip{a_i}{x}).
\end{equation}
Denoting $g \sim N(0,1)$ (the standard normal variable), we can compute the parameters in \eqref{eq: mu sigma eta} 
as follows: 
$$
\mu = \E f(g\|x\|) g, \quad \s^2 \le 1, \quad \eta = 1.
$$
(The computation for $\sigma$ follows by replacing the variance by the second moment.)
Substituting this into \eqref{eq: cone simplified}, we obtain the following result.

\begin{corollary}[Estimating from binary observations]		\label{cor: binary}
  Assume that $x \in K$ where $K$ is a fixed cone in $\R^n$.
  Assume the observations $y_i$ satisfy \eqref{single-index model} and 
  \eqref{eq: binary observations}. 
  Let $\bar{x} = x/\|x\|_2$.
  Then
  $$  
  \E \big\| \xhat -  \mu \bar{x} \big\|_2 \le C \, \frac{w_1(K)}{\sqrt{m}} 
  $$
  where $\mu = \E f(g\|x\|) g$ and $C = \sqrt{2\pi-4}+2 \approx 3.51$.   
\end{corollary}

Binary observations are important in both signal processing and statistics. We outline the connections between our model and the literature below.

\medskip

{\bf 1-bit compressed sensing.} The noiseless \textit{1-bit compressed sensing} model \cite{1-bit_CS} takes
observations 
$$
y_i = \sign(\ip{a_i}{x}).
$$
These form a particular case for which Corollary~\ref{cor: binary} applies with $\mu = \E |g| = \sqrt{2/\pi}$.    
This type of observations is motivated as a way of understanding the combination of extreme quantization with low-dimensional, or structured signals.  The $s$-sparse signal set described in Section~\ref{ssec: sparse} is of key interest.  We note that the magnitude of $x$ is completely lost in the observations (even with the non-linearity known), and thus it is standard in this literature to estimate the direction $\bar{x}$.  We may normalize our estimator to give
$$
\E \twonorm{ \frac{\xhat}{\twonorm{\xhat}} -  \bar{x} } \le C \, \sqrt{\frac{s \log(2n/s)}{m}}.
$$
This recovers the following result from the literature (see \cite{PV_IEEE,num_meas_sparse}):
$m \sim s \log(2n/s)$ 1-bit observations are sufficient to estimate an $s$-sparse vector in $\R^n$.

\medskip

{\bf Logistic regression.} \emph{Logistic regression} is a common statistical model for binary data, and takes the form
\[y_i = \sign(\ip{a_i}{x} + \e_i)\]
where $\e_i$ is logit noise.  We note that other forms of noise lead to other binary statistical models.  For example, if $\e_i$ is Gaussian, this recovers the \textit{probit} model.  There is a recent influx of statistical literature on combining sparsity with binary observations, see \cite{Bach} and references therein.  The standard method is $\ell_1$-penalized maximum likelihood estimation.  To perform maximum likelihood estimation, it is vital to know the likelihood function---this is equivalent to knowing the form of the noise.  However, in practice, given binary observations, it is often unclear which binary model to choose, and one is chosen arbitrarily.  In the present paper we emphasize that estimation can be done accurately without precise knowledge of relationship between $\ip{a_i}{x}$ and $y_i$.


\section{Optimality}
\label{sec: optimality}
In this section we determine general conditions under which the projection estimator gives an optimal result up to a numerical constant, i.e., there is no estimator which significantly improves on the projection estimator.  We begin by considering the noisy linear model described in \eqref{eq: noisy linear observations}. This will automatically give a lower bound in the case when the observations include a non-linearity. We will come to the following intriguing conclusion.
\begin{quote}
{\em When the measurements are noisy, an unknown, noninvertible non-linearity in the measurements often does not significantly decrease one's ability to estimate the signal.}
\end{quote}


\subsection{Lower bound in the linear model}
\label{ssec: linear model}

We begin by considering the linear model with Gaussian noise
\begin{equation} \label{eq: linear model Gaussian noise}
y_i = \ip{a_i}{x} + \e_i
\end{equation}
where the noise variables $\e_i \sim N(0, \nu^2)$ are independent of each other and the $a_i$. (The parameter $\nu$ is the level of noise.)  For compact notation, let $y \in \R^m$ be the vector of observations and $A \in \R^{m \times n}$ be the matrix whose $i$-th row is $a_i^\tran$.

Our goal is to determine conditions on $K$ and the noise which imply that the projection estimator is minimax up to a numerical constant.  For simplicity, we refer to numerical constants as $C$ if they are greater than 1, and $c$ if they are less than 1.  These constants may change from instance to instance, but each is bounded by an absolute numerical value.  We refer to the projection estimator \eqref{eq: xhat} as $\xhat_{\text{proj}}$ in this section.


The {\em local packing number}, which we define as follows, plays a key role in estimation error.
\begin{definition}[Local packing number, $P_t$]
Given a set $K \subset \R^n$, the local packing number $P_t$ is the packing number%
\footnote{Given a set $K \subset \R^n$ and a scalar $t > 0$, a packing of $K$ with balls of radius $t$ is a set $\XX \subset K$ satisfying $\twonorm{v - w} \geq t$ for each pair of distinct vectors $v,w \in \XX$.  The packing number is the cardinality of the largest such packing.}
of $K \cap t B_2^n$ with balls of radius $t/10$.
\end{definition}

We now give a lower bound on the minimax error in the noisy linear model.
\begin{theorem}\label{thm:linear lower bound}
Assume that $x \in K$ where $K$ is a star-shaped subset of $\R^n$.  Assume that the observations $y_i$ are given by \ref{eq: linear model Gaussian noise}.
 Let 
\begin{equation}\label{eq: def delta}
\delta_* := \inf_{t > 0} \left\{ t + \frac{\nu}{\sqrt{m}}\left[1 + \sqrt{\log P_t}\right]\right\}.
\end{equation}
Then there exists an absolute constant $c > 0$ such that any estimator $\xhat$ which depends only on the observations $y_i$ and measurements $a_i$ satisfies
\[\sup_{x \in K} \E \twonorm{\xhat - x} \geq c \min(\delta_*, \diam(K)).\]
\end{theorem}

Let us compare this to the upper bound achieved by the projection estimator, as derived in Corollary \ref{cor: noisy linear}:
\[\E \twonorm{\xhat_{\text{proj}} - x} \leq C \inf_{t > 0} \left\{t + \frac{\nu + \twonorm{x}}{\sqrt{m}}\Big[1 + \frac{w_t(K)}{t}\Big]\right\}. \]
We now determine conditions under which the two match up to a constant.  Observe that the two match most closely when the noise is large.  In particular, when $\nu \geq \twonorm{x}$, we have the simplified error bound:
\[\E \twonorm{\xhat_{\text{proj}} - x} \leq C' \inf_{t > 0} \left\{t + \frac{\nu}{\sqrt{m}}\Big[1 + \frac{w_t(K)}{t}\Big]\right\} =: C' \delta^*. \]
Further, since the estimator projects onto $K$, the upper bound may be tightened:
\[\E \twonorm{\xhat_{\text{proj}} - x} \leq C' \min(\delta^*, \diam(K)).\]
The only difference between lower and upper bounds is that $\sqrt{\log P_t}$ is replaced by $w_t(K)/t$.  
While these quantities may be in general different, for many sets they are comparable. 
To compare them, let us introduce the following geometric parameter:

\begin{definition}[Ratio: packing to local mean width]
Define $\alpha$ as
\begin{equation} \label{eq: ratio}
\alpha = \alpha(K) = \sup_{t > 0} \frac{w_t(K)}{t \sqrt{\log P_t}}.
\end{equation} 
\end{definition}
It follows by definition of $\alpha$ that the lower bound of Theorem \ref{thm:linear lower bound} satisfies
\[\delta_* \geq \delta^*/\alpha.\]
The following corollary follows immediately.
\begin{corollary}\label{cor:linear lower bound}
Assume that $x \in K$ where $K$ is a star-shaped subset of $\R^n$.  Assume that the observations $y_i$ are given by \ref{eq: linear model Gaussian noise}.  Let 
\begin{equation}
\delta^* :=  \inf_t \big\{ t + \frac{\nu}{\sqrt{m}}\Big[1 + \frac{w_t(K)}{t}\Big]\big\}.
\end{equation}
Then any estimator $\xhat$ satisfies
\[
\sup_{x \in K} \E \twonorm{\xhat - x} \geq c_\alpha \min(\delta^*, \diam(K)),
\]
where $c_\alpha = c/\alpha$ for the numerical constant $c$ defined in Theorem~\ref{thm:linear lower bound} and $\alpha$ defined in \eqref{eq: ratio}.
\end{corollary}

Thus, in the high-noise regime, the difference between upper and lower bounds is a factor of (numerical constant times) $\alpha$.  Fortunately, for many sets of interest, $\alpha$ is itself bounded (on both sides) by a numerical constant.  In particular, this holds true for sparse vectors, low-rank matrices and the $\ell_1$-ball, as described in Section~\ref{ssec: reverse sudakov}.  Thus, in this case, the minimax error satisfies
\[ c \min(\delta^*, \diam(K)) \leq \inf_{\xhat} \sup_{x \in K} \E \twonorm{\xhat - x} \leq C \min(\delta^*, \diam(K)).\]

\begin{remark}[Defining $\alpha$ at scale]
\label{rem: alpha at scale}
If desired, $\alpha$ can be tightened by defining it at the relevant scale.  Let $\delta^*_2 = \frac{1}{2} \delta^*$.  Then one can redefine $\alpha$ as 
\begin{equation}\label{eq: alpha at scale}
\alpha(K) = \frac{w_{\delta^*_2}(K)}{\delta^*_2 \sqrt{\log P_{\delta^*_2}} }
\end{equation}
and the result of the Corollary still holds.  
\end{remark}

\begin{remark}[General measurement vectors]
Theorem \ref{thm:linear lower bound} and Corollary \ref{cor:linear lower bound} hold for a general class of random (or deterministic) measurement vectors.    The theory only requires that $A$ does not stretch out signals too much in any one direction,  that is,
\[\frac{1}{m} \|\E A^* A\| \leq 1.\]
Note that for the special case of a Gaussian measurement matrix $\frac{1}{m} \E A^* A = I_n$ and the above requirement is met; similarly if $A$ contains i.i.d.~entries with variance 1, or, more generally, if the rows are in \textit{isotropic position}.  If one wishes to generalize further, a look at the proof implies that we only require
\[\E \sup_{x \in K - K} \frac{\twonorm{A x}}{\twonorm{x}} \leq 1.\]
Further, the number 1 on the right-hand side could be replaced by any numerical constant without significantly affecting the result.
\end{remark}

\begin{remark}[Signal-to-noise ratio assumption]
Our optimality assumption requires $\twonorm{x} \leq \nu$, which places a bound on the norm of the signal.  To give a precise minimax accounting of the error, this assumption can be incorporated into Theorem \ref{thm:linear lower bound} by replacing $K$ by $K \cap \nu B_2^n$.  In this case, the theorem implies the lower bound
\[\E \twonorm{\xhat - x} \geq c\min(\delta, \diam(K), \nu).\]
Of course, if $\nu$ is known then this information can be taken into account in the projection estimator, giving it a mirroring error bound.  One may check that the geometric assumption on $K$ need not be adjusted.  This follows from a rescaling argument and since $K$ is star-shaped.
\end{remark}

\subsection{Optimality in the non-linear model}
\label{ssec: upper vs lower}

Now we extend our discussion of optimality to non-linear observations. 
Let us now assume that the noisy linear data is passed through an unknown non-linearity:
\begin{equation}\label{eq: nonlinear model}
y_i = f(\ip{a_i}{x} + \e_i)
\end{equation}
where $\e_i \sim N(0, \nu^2)$.  Note that the non-linearity can only decrease one's ability to estimate $x$, and thus our lower bound still holds.  It remains to determine conditions under which our upper bound matches.  For easiest comparison, we rescale Theorem \ref{thm: main} to give the reconstruction error in estimating $x$.

\begin{theorem}[Non-linear estimation rescaled]				\label{thm: main rescaled}
   Let $\lambda := \twonorm{x}/\mu$, so that $\E \lambda \xlin = \lambda \mu \bar{x} = x$.  Assume that $x \in K' = \lambda K$ where $K$ is a fixed star-shaped closed subset of $\R^n$. 
  Then the non-linear estimator $\xhat$ defined in \eqref{eq: xhat} satisfies the following for every $t>0$:
  \begin{equation}         
  \E \|\lambda \xhat - x\|_2 \le t + \frac{2 \lambda }{\sqrt{m}} \Big[\s + \eta \, \frac{w_t(K')}{t} \Big].
  \end{equation}
  Here $\mu$, $\sigma$ and $\eta$ are the numbers defined in \eqref{eq: mu sigma eta}.
\end{theorem}
\begin{proof}
The unscaled version, Theorem \ref{thm: main}, gives
\[\E \|\xhat - \mu \bar{x}\|_2 \le t + \frac{2}{\sqrt{m}} \Big[\s + \eta \, \frac{w_t(K)}{t} \Big].\]
Now multiply both sides of the inequality by $\lambda$, substitute $ t $ for $\lambda t$, and note that 
\[
w_{\lambda t}(\lambda K) = \lambda w_t(K)
\]
to complete the rescaling argument.
\end{proof}

Comparing to the estimation error from noisy linear observations \eqref{eq: noisy linear}, we find that the non-linearity increases the error by at most a constant factor provided
\[
\lambda (\s + \eta) \leq C (\twonorm{x} + \nu).
\]
Below, we give some general conditions on $f$ which imply this inequality for all signals $x$ and noise levels $\nu$.

\begin{lemma}[Conditions on the non-linearity] 
\label{lem: conditions on f}
Suppose that $f$ is odd and non-decreasing.  Further, suppose that $f$ is sub-multiplicative on $\R^+$, that is $f(a \cdot b) \leq f(a) \cdot f(b)$ for all $a, b > 0$.  Then, for the model defined in \eqref{eq: nonlinear model}, the parameters $\lambda = \mu^{-1} \twonorm{x}$, $\sigma$, $\eta$ defined in \eqref{eq: mu sigma eta} and the noise level $\nu$ satisfy
\begin{equation} \label{eq: bound constants}
\lambda (\s + \eta) \leq C_f (\twonorm{x} + \nu),
\end{equation}
where
\[
C_f = C \E[f^4(g)]^{1/4}, \quad C = 48^{1/4}\Pr{\abs{g}\geq 1} \approx 1.8.
\]
with $g \sim N(0,1)$.  In particular, $C_f$ does not depend on $\twonorm{x}$ or $\nu$.
\end{lemma}

Below are a few examples of non-linearities $f$ that satisfy the conditions of Lemma~\ref{lem: conditions on f}:
\begin{enumerate}
\item The identity function, $f(x) = x$; $C_f \approx 2.36$.
\item The sign function, $f(x) = \sign(x)$; $C_f \approx 1.8$.
\item Monomials of the form $f(x) = x^k$ for some odd natural number $k$; $C_f \approx 1.8 M_{4k}$, where $M_{4k} = [(4k-1)!!]^{1/4k}$ is the $4k$th moment of the standard normal distribution.
\end{enumerate}

\begin{remark}[Linear combinations]
Note that the left-hand side of \eqref{eq: bound constants} remains unchanged if $f$ is multiplied by a constant. Moreover, take $k$ functions $f_1, \dotsc, f_k$ satisfying the conditions of Lemma~\ref{lem: conditions on f}, and consider a nonlinear model of the form
\[
y_i = \sum_{j=1}^k C_j f_j(\ip{a_i}{x} + \e_i), \quad C_j > 0\, \forall j
\]
where all $C_j$ are positive (so that $C_j f_j$ is still monotonic). It is not hard to see that $\lambda$, $\sigma$, $\eta$, and $\nu$ satisfy
\[
\lambda(\sigma + \eta) \leq \sqrt{k} \max_j C_{f_j} (\twonorm{x} + \nu),
\]
where $C_{f_j}$ are defined as in Lemma~\ref{lem: conditions on f}.
This allows us to obtain uniform error bounds for finite linear combinations of any of the functions satisfying the conditions of Lemma~\ref{lem: conditions on f}, such as odd nondecreasing polynomials of degree at most $k$.
\end{remark}

We summarize our findings in the following corollary:
\begin{corollary}[Summary or optimality results]\label{coroll: summary}
Consider the nonlinear model \eqref{eq: nonlinear model}, and suppose the nonlinearity takes the form $f = \sum_{j=1}^k C_j f_j$, where the functions $f_j$ are odd, nondecreasing, and sub-multiplicative. Let $C_{f_j}$ be the constants defined in Lemma~\ref{lem: conditions on f}. Let the signal set $K$ satisfy $\alpha(K) < C_1$, for $\alpha$ defined in \eqref{eq: alpha at scale}. Suppose the noise level $\nu$ is comparable to the signal level $\twonorm{x}$: $\nu \geq c_2 \twonorm{x}$.  Define
\[
\delta^* = \inf_t \big\{ t + \frac{\nu}{\sqrt{m}}\Big[1 + \frac{w_t(K)}{t}\Big]\big\}.
\]
Let $\xhat_{\text{proj}}$ be the projection estimator defined in \eqref{eq: xhat}, and let $\xhat_{\text{opt}}$ be the optimal estimator of $x$. Then there exists an absolute numeric constant $c$, and constants $C(f_j)$ depending only on $f_j$, such that
\[
\twonorm{\xhat_{\text{opt}} - x}^2 \geq c \frac{c_2}{C_1} \min(\delta^*, \diam(K))
\]
and
\[
\twonorm{\xhat_{\text{proj}} - x}^2 \leq \sqrt{k} (\sum_{j=1}^k C^2_{f_j}) \min(\delta^*, \diam(K)).
\]
\end{corollary}


\subsection{Sets satisfying the Geometric Condition}
\label{ssec: reverse sudakov}

In this section we show that for the three commonly used signal sets discussed in Sections~\ref{ssec: sparse}--\ref{ssec: l1}, the relationship
\begin{equation}\label{eq: geometric condition}
\frac{w_t(K)}{t \sqrt{\log P_t}} \leq C
\end{equation}
holds for all resolutions $t$ (with $t < 1$ in the case of the $\ell_1$-ball), for some universal constant $C$. Thus, $\alpha(K) \leq C$ for these sets. The condition~\eqref{eq: geometric condition} is essentially asserting that \emph{Sudakov's minoration inequality} (see \cite[Theorem 3.18]{LT}) is \emph{reversible} for $(K-K) \cap tB^n_2$ at scale $t/10$. Indeed, applying Sudakov's minoration inequality to the Gaussian process $X_x = \ip{g}{x}$, $g \sim N(0,I_n)$, defined on $(K-K) \cap t B_2^n$ shows $\E \sup X_x = w_t(K) \geq c t \sqrt{\log P_t}$ for some numeric constant $c$. (Note that the packing number of $(K-K) \cap t B^n_2$ is at least as great as the packing number of $K \cap t B^n_2$.)

\medskip

{\bf Sparse vectors.} Let $S_{n,s} \subset \R^n$ be the set of $s$-sparse vectors, i.e., the set of vectors in $\R^n$ with at most $s$ nonzero entries. Note that $S_{n,s}$ is a cone, so if \eqref{eq: geometric condition} is satisfied for $S_{n,s}$ at some level $t$, then it is satisfied at all levels. Thus, without loss of generality we take $t = 1$.

For $K = S_{n,s}$ we have $K-K = S_{n,2s}$. As shown in Section \ref{ssec: sparse},
\[
w^2_1(K) = w^2(S_{n,2s} \cap B_2^n) \leq C s \log\frac{2n}{s},
\]
for some constant $C$. On the other hand, we show that there exists a $1/10$-packing $\XX_1$ of $S_{n,s} \cap B_2$ such that
\[
\log \abs{\XX_1} \geq c\sqrt{s \log\frac{2n}{s}},
\]
for some (possibly different) constant $c$. Such a packing exists when $s > n/4$, because then $S_{n,s}$ contains an $(n/4)$-dimensional unit ball, whose packing number is exponential in $n$. Thus, we may assume $s < n/4$. Consider the set $Q$ of vectors $x$, where each $x \in Q$ has exactly $s$ nonzero coordinates, and the nonzero coordinates are equal to exactly $s^{-1/2}$. Thus, $Q \subset S_{n,s} \cap B_2$, and $\abs{Q} = \binom{n}{s}$. $Q$ itself is not the packing, but we will show that we can pick a large subset $\XX_1$ of $Q$ such that any two vectors $x, y \in \XX_1$ satisfy $\twonorm{x-y} > 1/10$.

Consider picking vectors from $Q$ uniformly at random. For two uniformly chosen vectors $x, y \in Q$, we have
\[
\Pr{\twonorm{x-y}^2 \leq \frac{1}{100}} = \Pr{\text{$x$, $y$ disagree on at most $\frac{s}{100}$ coordinates}}.
\]
This requires $y$ to have at least $\frac{99}{100}s$ of the same nonzero coordinates as $x$. Given $x$, and assuming without loss of generality that $0.01s$ is an integer (rounding will not make a significant effect in the end), this happens for exactly
\[
\binom{s}{0.99s}\binom{n-0.99s}{0.01s} \text{ out of } \binom{n}{s}
\]
values of $y \in Q$. Using Stirling's approximation (and the assumption $s < n/4$),
\[
\Pr{\twonorm{x-y} \leq \frac{1}{10}} \leq \exp\left(-C s \log\frac{2n}{s}\right).
\]
Now let $\XX_1 \subseteq Q$ contain $c \exp(s \log\frac{2n}{s})$ uniformly chosen elements of $Q$. Then with positive probability $\twonorm{x-y} > 1/10$ for all pairs $x, y \in \XX_1$. In particular, there exists at least one $1/10$-packing $\XX_1$ of $S_{n,s} \cap B_2$ of size $c \exp(s \log\frac{2n}{s})$.

\medskip

{\bf Low-rank matrices.}
Let $M_{d_1, d_2,r}$ be the set of $d_1 \times d_2$ matrices with rank (at most) $r$; clearly $r \leq \min(d_1,d_2)$. This is again a cone, so it suffices to show that \eqref{eq: geometric condition} is satisfied at a single level $t = 1$. Now, for $K = M_{d_1, d_2 ,r}$, we have $K-K = M_{d_1, d_2 ,2r}$, so (see \cite[Section 3.3]{PV_IEEE} or Section~\ref{ssec: low-rank} of this paper)
\[
w_1^2(K-K) = w^2(M_{d_1, d_2, 2r}) \cap B^{d_1 d_2}_2) \leq C (d_1 + d_2)r.
\]

For the packing number, note that $M_{d_1, d_2,r}$ contains all matrices whose last $d_2-r$ rows are identically zero, and also all matrices whose last $d_1 - r$ rows are identically zero. Thus, $M_{d_1,d_2,r}$ contains a copy of $\R^{d_1r}$ and a copy of $\R^{d_2r}$. Since the packing number of a unit ball in Euclidean space is exponential in its dimension, we conclude for the packing number of $K$
\[
\log P_1 \geq c \max(d_1,d_2) r \geq \frac{c}{2}(d_1+d_2) r,
\]
as required.

\medskip

{\bf Approximately sparse signals: $\ell_1$-ball.}
For the set of approximately sparse signals contained in the $\ell_1$-ball $B_1^n$, we will show that condition~\eqref{eq: geometric condition} is satisfied at all levels $t < 1$. Note that for $K = B_1^n$, $K-K \subset 2 B_1^n$. We first derive precise bounds on the local mean width of $K$,
\[
w_t^2(B_1^n) \leq w^2(2B_1^n \cap t B_2^n) \leq t^2 w^2(2t^{-1} B_1^n \cap B_2^n).
\]
It is easy to see that, for all $t$, $w_t^2(B_1^n) \leq C'\min(t^2 n, \log n)$. The former will apply when $t$ is very small ($t \leq C n^{-1/2}$); the latter will apply when $t$ is very large ($t \geq c$). In the intermediate regime $C n^{-1/2} < t < c$, let $\sqrt{s} = 2t^{-1}$; then $c^{-2} < s < C^{-2}n$. Assume for simplicity that $s$ is an integer; rounding $s$ will not affect the results substantially. Now, by \cite[Lemma 3.1]{PV_CPAM}
\[
\frac12(\sqrt{s} B_1^n \cap B_2^n) \subset \conv(S_{n,s} \cap B_2^n),
\]
and by \cite[Lemma 2.3]{PV_IEEE},
\[
w^2 \conv(S_{n,s} \cap B_2^n) = w^2(S_{n,s} \cap B_2^n) \leq C' s \log\frac{2n}{s}.
\]
Consequently,
\[
t^{-2} w_t^2(B_1^n) \leq
\begin{cases}
C' n, & t \leq C n^{-1/2};\\
C' t^{-2} \log(nt), & C n^{-1/2} < t < c;\\
C' t^{-2} \log n, & c \leq t \leq 1.
\end{cases}
\]

We now consider packings of $B_1^n \cap t B_2^n$ by balls of radius $t/10$. Clearly, for $t < n^{-1/2}$, the packing number is at least exponential in $n$, because the set contains an $n$-dimensional ball of radius $t$; and for $c < t < 1$ the packing number is at least $2n$, because we can simply pack the vertices. It remains to treat the intermediate regime $n^{-1/2} < t < c$. Note that the problem is equivalent to packing $t^{-1}B_1^n \cap B_2^n$ by balls of radius $1/10$. Let $\sqrt{s} = t^{-1}$; assuming $s$ is an integer, we have $S_{n,s} \subset t^{-1} B_1^n \cap B_2^n$, so we may apply the packing bounds for sparse vectors:
\[
\log \abs{\XX_t} \geq c s \log\frac{2n}{s} = c t^{-2} \log(nt).
\]
It is not hard to put these bounds together to conclude that indeed,
\[
\log \abs{\XX_t} \geq
\begin{cases}
c n, & t \leq C n^{-1/2};\\
c t^{-2} \log(nt), & C n^{-1/2} < t < c;\\
c t^{-2} \log n, & c \leq t \leq 1.
\end{cases}
\]

\section{Related results: Low $M^*$ estimate, Chatterjee's least squares estimate}

\subsection{Low $M^*$ estimate from geometric functional analysis}\label{sec: comparison to M^*}

The function 
$$
t \mapsto \frac{w_t(K)}{t}
$$
that appears in Theorem~\ref{thm: main} has been studied in geometric 
functional analysis. It is known to tightly control  
the {\em diameter of random sections} of $K$. 
An upper bound on the diameter is known as the {\em low $M^*$ estimate}
(see \cite{Milman1,Milman2,Pajor-Tomczak,Gordon})
and lower bounds have been established in \cite{GM1,GM2,GM3,GM4}. 
The following is a simplified version of the low $M^*$ estimate, see \cite{MPT}.

\begin{theorem}[Low $M^*$ estimate]
  There exists an absolute constant $c>0$ such that the following holds. 
  Let $K$ be a star-shaped subset of $\R^n$.
  Let $A$ be an $n \times m$ random matrix whose entries are independent $N(0,1)$ random variables.
  Assume that $t>0$ is such that 
  $$  
  \frac{w_t(K)}{t} \le c \sqrt{m}.
  $$  
  Then with probability at least $1-\exp(-cm)$, we have 
  \begin{equation}         \label{eq: low M*}
  \|u-v\|_2 \le t \quad \text{for any } u, v \in K, \; u - v \in \ker(A).
  \end{equation}
\end{theorem}
The conclusion \eqref{eq: low M*} can be interpreted in a geometric way as follows.
All sections of $K$ parallel to the random subspace $E = \ker(A)$ have diameter at most $t$, i.e.
$$
\diam \big( K \cap (E+v) \big) \le t  \quad \text{for all } v \in \R^n.
$$

The relevance of the low $M^*$ estimate to the estimation problems was first observed in \cite{MPT}.
Suppose one wants to estimate a vector $x \in K$ from linear observations as in 
Section~\ref{s: linear observations}, i.e.
$$
y_i = \ip{a_i}{x}, \quad i=1,\ldots,m.
$$
Choose an arbitrary vector $\xhat_0 \in K$ which is consistent with the observations, 
i.e. such that
$$
\ip{a_i}{\xhat_0} = y_i, \quad i=1,\ldots,m.
$$
For convex feasible sets $K$, computing a vector $\xhat_0$ this way 
is an algorithmically tractable convex feasibility problem. 

To bound the error of this estimate, we can apply the low $M^*$ estimate \eqref{eq: low M*} for 
$A$ the $m \times n$ matrix with rows $a_i^\tran$.
Since $x, \xhat_0 \in K$ and $x - \xhat_0 \in \ker(A)$, it follows that with high probability,
$\|\xhat_0 - x\|_2 \le t$. To summarize, 
\begin{equation}         \label{eq: low M* error}
\frac{w_t(K)}{t} \le c \sqrt{m} \quad \text{implies} \quad \|\xhat_0 - x\|_2 \le t \text{ with high probability}.
\end{equation}
In particular, if $K$ is a cone then $w_t(K)/t = w_1(K)$, and we can let $t \to 0$. 
In this case, the low $M^*$ estimate guarantees {\em exact recovery} once 
$w_1(K) \le c \sqrt{m}$. 

To compare \eqref{eq: low M* error} with the error bound of the non-linear estimation,
we can state the conclusion of Theorem~\ref{thm: main} as follows:
$$
\frac{w_t(K)}{t} \le \e \sqrt{m} \quad \text{implies} \quad \|\xhat - x\|_2 \le t + 2 \e + \frac{2\s}{\sqrt{m}} 
\text{ with high probability}.
$$
The two additional terms in the right hand side can be explained by the fact that 
the the exact recovery is impossible in the single-index model \eqref{single-index model}, in particular because of 
the noise and due to the unknown non-linear dependence of $y_i$ on $\ip{a_i}{x}$.

\subsection{Chatterjee's least squares under convex constraint}

As this paper was being written, \cite{Chatterjee} proposed a solution to 
a closely related problem. Suppose that an unknown vector $x \in \R^n$ lies in a known closed convex set $K$. 
We observe the noisy vector 
$$
y = x + g, \quad \text{where} \quad g \in N(0,I_n).
$$
We would like to estimate $x$ from $y$.
This is very similar to the linear estimation problem \eqref{eq: noisy linear observations}, 
which can be written in the form $y = Ax + g$.
An important difference is that in Chatterjee's model, $A$ is the identity matrix; 
so one needs to take $n$ observations of an $n$-dimensional vector $x$ (given as coordinates $y_i$ of $y$). 
In contrast, the current paper assumes that $A$ is an $m \times n$ Gaussian matrix (essentially a projection), 
so one is allowed to take $m$ observations where usually $m \ll n$.

The estimator proposed in \cite{Chatterjee} is the least-squares estimator, which clearly equals 
the metric projection of $y$ onto $K$, that is 
$$
\xhat = P_K(y).
$$
Note that this coincides with the second step of our estimator. (The first step---a linear estimator---is clearly not needed 
in Chatterjee's model where $A$ is identity.)

The performance of the least-squares estimator is quantified in \cite{Chatterjee} using a version of the local mean width
of $K$; this highlights one more similarity with the current paper.

\section{Connections to statistics and econometrics} \label{sec: elena's section}

The single-index model we consider is widely used in econometrics; the monograph of \cite{Horowitz} gives a good introduction to the subject. In this section, we discuss connections and differences between our method and some of the literature. For this section only, we will use notation more common in statistics literature; Table~\ref{table} below summarizes the translation between the two worlds.

\begin{table}[ht]
\caption{Dictionary of notation} \label{table}

\begin{tabular}{l|c|c}
\hline\noalign{\smallskip}
Quantity & Rest of paper & This section\\
\noalign{\smallskip}\hline\noalign{\smallskip}
independent variable & $a_i \in \R^n$ & $X_i \in \R^p$\\
observation & $y_i \in \R$ & $Y_i \in \R$\\
index & $x \in \R^n$ & $\beta \in \R^p$\\
dimension of parameter space & $n$ & $p$\\
number of observations & $m$ & $n$\\
\end{tabular}
\end{table}

We note that much of the work on single-index models considers a broader formulation, namely
\[
Y_i = f(\ip{\beta}{X_i}) + \e_i, \qquad \E[\e_i \vert X_i] = 0.
\]
Our work relies crucially on the additional structural assumption that $\e_i$ and $X_i$ are conditionally independent given $\ip{\beta}{X_i}$.

The single-index model is similar to generalized linear models (GLMs) in regression analysis, and contains many of the widely used regression models; for example, linear, logistic, or Poisson regressions. Our results apply directly to those models, with the added advantage that we do not need to know in advance which of the models is appropriate for a specific dataset.
(Note especially that Corollary~\ref{coroll: summary} guarantees uniform error bounds e.g., for most cubic functions $f$.) 
In addition, Section~\ref{ssec: binary} demonstrates the optimality of our approach for a rather general model of sparse binary-response data, which encompasses the logit or probit models for specific distributions of $\e_i$.

However, we remark that there is no containment between single-index models and GLMs; for example, the negative binomial model is a commonly used GLM which does not fit in the single-index framework. (There is also no reverse containment, as single-index models do not require the distribution of $Y_i$ to come from an exponential family.)

The literature on single-index and more general semiparametric models is large and growing, so we can only outline some of the connections here. The primary appeal of single-index models to statisticians has been that it allows estimation of $\beta$ at the parametric rate ($\sqrt{n}(\hat \beta_n - \beta)$ is tight) as well as estimation of $f$ at the same rate as if the one-dimensional input $\ip{\beta}{X_i}$ were observable. Estimation of an unknown function is notoriously difficult when the domain is high-dimensional, so the latter feature is particularly attractive when $p$ is large. We note that much of the analysis of these results is focused on the asymptotic rates of convergence, although some recent papers \cite{Hristache,Alquier,Tsybakov} also provide nonasymptotic (finite $n$) guarantees.

However, in contexts where the primary concern is with the index $\beta$, the classical approaches (such as \cite{Stoker,Powell,Hristache}) can be unduly restrictive with regards to the link function. This is because they tend to be based on the \emph{average derivative} method, i.e. the observation that
\[
\frac{\del \E[Y_1 \vert X_1 = x]}{\del x} = C \beta.
\]
Naturally, methods based on this idea require the link function $f$ to be (at least) differentiable, whereas we can allow $f(\cdot) = \sign(\cdot)$.

To our knowledge, the only fundamentally different approach to recovery of $\beta$ alone (without $f$) was taken by Li and Duan \cite{Li_Duan}. The authors considered observations $X_i$ with an elliptically symmetric distribution (e.g., correlated normal). They then demonstrated under some mild conditions that any method based on minimization of a convex criterion function (e.g., least squares) would achieve one of its minima at $\beta^* \propto \beta$. Thus, essentially any method for which the minimizer happens to be unique can be used to consistently recover (a multiple of) $\beta$; under some additional assumptions, $\sqrt{n}(\hat \beta_n - \beta)$ is asymptotically normal. The major advantage of this method is that it does not in any way rely on the smoothness of the link function $f$.

The above discussion applies to low-dimensional settings, in which all that is known about $\beta$ is $\beta \in \R^n$ (perhaps with $\twonorm{\beta} = 1$). Recently, a lot of work has been devoted to the analysis of high-dimensional data ($p > n$) with additional structural assumptions about $\beta$, for example sparsity. In the case of linear link function $f$, sparse high-dimensional regression is well-studied (see \cite{vdGeer} and references therein). On the other end of the complexity spectrum, the work of Comminges and Dalayan \cite{Comminges} considers a sparse nonparametric model
\[
y_i = f(X_i) = f_J(\{(X_i)_j: j \in J\}), \quad \abs{J} \leq s
\]
where at most $s$ of the components of the $X_i$ are relevant to the model. They ask the question of when the set $J$ of relevant components can be recovered, and find that the necessary number of measurements is
\[
n \geq C_1 \exp(C_2 s) \log(p/s).
\]
We compare this to our result on recovering a single index $\beta$ belonging to the cone $K$ of $s$-sparse vectors, for which
\[
(w_1(K))^2 \leq s \log (p/s).
\]
We see that going from a single-index to a nonparametric model involves an exponentially larger number of measurements (but the number is exponential in the underlying sparsity, not in the full dimension of the parameter space).

We mention also some of the recent work on the estimation of the link function $f$ in high-dimensional settings. Alquier et al. \cite{Alquier} and Cohen et al. \cite{Cohen} consider the model of a link function $f$ which depends on at most $s$ of the components of $x$. More generally, Dalayan et al. \cite{Tsybakov} allow all link functions that can be decomposed as the sum of at most $m$ ``nice'' functions $f_V$, with each $f_V$ depending on at most $s$ coordinates of $x$. (A special case of all of these models, including ours, is $f(x) = \tilde f(\ip{x}{\beta})$, where $\beta$ $s$-sparse and $\tilde f$ comes from a nice function class.) The recent work of Dalayan et al. \cite{Tsybakov} obtain general bounds on the quality of estimation of $f$, uniformly over classes containing $f_V$.

\subsection{Connection to statistical learning theory}
In statistical (supervised) learning theory, the data is used to
determine a function that can be used for future prediction. Thus the
data $\{a_i, y_i, i = 1, \hdots, m\}$ is used to generate a function which
can predict $y_{m+1}$ given $a_{m+1}$. Such a predicting function is usually
sought in a given function class. A main focus of the research is to
develop oracle inequalities, which state that the predicting function
estimated from the data performs nearly as well as the optimal
function in that class.

By restricting to classes of linear functionals, such results can
(with some work) be specialized and translated to give error bounds on
the estimation of $x$, as in our paper.  We focus on the results of \cite{lecue2013learning} which gives general error bounds under mild conditions. Nevertheless, there are important
differences between these and the results of our paper.  First,
statistical learning theory concentrates on \textit{empirical risk
minimization} \cite{lecue2013learning}, which, when combined with a \textit{squared loss
function}, and specialized to linear functionals, recovers a
generalized version of the Lasso. In contrast, we consider a simpler,
and often more efficient, projection method in this paper.  (We note
in passing that, after writing this paper, the first two authors
studied the behaviour of the Lasso under the single-index model
\cite{plan2015generalized}.) Furthermore, we make especially mild assumptions on the
data.  
In
contrast to \cite{lecue2013learning}, our Theorem \ref{thm: main} does not require
$y_i$ to be sub-Gaussian.  Instead, we roughly only require $y_i$
to have finite second moment so that the parameters $\mu, \sigma$ and
$\eta$ are well defined.  On the other hand, the Lasso can return an exact solution in the noiseless model, whereas the project-based estimator proposed in this paper always has some error attached.  The Lasso can also tolerate the case when the measurement vectors are anisotropic \cite{plan2015generalized}.

We also point to the work of \cite{negahban2012unified} which can also be specialized to linear functionals to give an error bound similar to the one in our paper.  However, the simple projection method espoused in our paper is often more computationally efficient than the optimization programs suggested in \cite{negahban2012unified}.  Further, the conditions on the signal structure are much milder.  Indeed, in contrast to \cite{negahban2012unified}, there is no need for \textit{decomposability} or \textit{restricted strong convexity}.  

Finally, we note that the behavior of the generalized Lasso under the linear model is well studied, and error bounds are known with sharp constants \cite{oymak2013squared}.  Recently, building on the work of \cite{plan2015generalized}, sharp constants have been given for the non-linear model in the asymptotic regime \cite{thrampoulidis2015lasso}.  It is an interesting open question to see if this can also be done in the non-asymptotic regime.  

\section{Discussion}
\label{sec: discussion}

We have analyzed the effectiveness of a simple, projection-based estimator in the semiparametric single index model.  We showed that by projecting onto the set encoding the signal structure, rough, unspecified observations could be combined to give an accurate description of the signal.  The gain from this dimension reduction was described by a geometric parameter---the mean width.  When the noise is large and under mild assumptions on the rest of the model, we showed that the estimator is minimax optimal up to a constant.  We came to the surprising conclusion that an unknown, noninvertible non-linearity often does not significantly effect one's ability to estimate the signal, aside from a loss of scaling information and an extra numerical constant in the error bound.

By comparing to what was known in the classic literature on the semiparametric single-index model, we believe our results a) give a simple method of testing whether estimation is possible based on easily computable parameters of the model, b) allow for non-invertible, discontinuous, and unknown linearities and c) give a careful accounting of the benefit of using a low-dimensional signal structure.  

Nevertheless, our model takes the measurement vectors to be standard normal, and it is important to understand whether this assumption may be generalized, and whether the theory in this paper matches practice.  We discuss this challenge in the rest of this section.  We pause to note that the lower bounds take a very general model for the measurement vectors, which may even be deterministic, thus it is the model used for the upper bounds which requires generalization. 
  
There is a simple key observation that gives the first step to the theory in this paper:  The linear estimator described in Section~\ref{ssec: linear estimation} is an unbiased estimate of the signal, up to scaling.  It may be surprising at first that this holds regardless of a non-linearity in the observations.  This fortunate fact follows from the Gaussianity of the measurements, as described in the next section.    

It is straightforward to generalize the theory to the case $a_i \sim N(0, \Sigma)$, provided $\Sigma$ is known or can be estimated.  One only needs to multiply the linear estimator by $\Sigma^{-1}$ to give an unbiased estimator; as long as $\Sigma$ is well-conditioned, our theory remains essentially unchanged.  The next question is the applicability of other random measurement ensembles.  We see two avenues towards this research:  1) As we illustrated with the matrix completion example of Section \ref{ssec: low-rank}, if the observations are linear, then the linear estimator can be unbiased for non-Gaussian measurement vectors.  All that is needed is for the measurement vectors to be in isotropic position, that is, $\E a_i a_i^* = I_n$.  This is a common assumption in the compressed sensing literature \cite{RIPless}.  Of course, a slight non-linearity can be accounted for by a Taylor series approximation.  2)  A central idea in high-dimensional estimation is that non-Gaussian measurement vectors often give similar behavior to Gaussian measurement vectors.  For example, this is made explicit with the universality phenomenon in compressed sensing.  Such arguments have been applied to a special case of this semiparametric single-index model in \cite{ALPV} focused on binary observations with non-Gaussian measurement vectors.  We believe that similar comparison arguments may be applied in the more general setting of this paper.  

In conclusion, we emphasize that the assumption of Gaussian measurements has allowed a very clean theory, with relatively few assumptions in a general model, and we hope that it can be a first step towards such theory with other models of the measurement vectors.

\section{Proofs of Proposition~\ref{prop: linear estimation} and Theorem~\ref{thm: main}}				\label{s: proofs}	


\subsection{Orthogonal decomposition}

The proof of Proposition~\ref{prop: linear estimation} as well as of our main result, Theorem~\ref{thm: main},
will be based on the orthogonal decomposition of the vectors $a_i$ along the direction of $x$
and the hyperplane $x^\perp$. More precisely, we express
\begin{equation}         \label{eq: orthogonal decomposition}
a_i = \ip{a_i}{\bar{x}} \bar{x} + b_i.
\end{equation}
The vectors $b_i$ defined this way are orthogonal to $x$. 
Let us record a few elementary properties of this orthogonal decomposition.

\begin{lemma}[Properties of the orthogonal decomposition]		\label{lem: orthogonal decomposition}
  The orthogonal decomposition \eqref{eq: orthogonal decomposition}
  and the observations $y_i$ satisfy the following properties:
  \begin{enumerate}[(i)]
    \item $\ip{a_i}{\bar{x}} \sim N(0,1)$;
    \item $b_i \sim N(0, I_{x^\perp})$;
    \item $\ip{a_i}{\bar{x}}$ and $b_i$ are independent;
    \item $y_i$ and $b_i$ are independent.
  \end{enumerate}
\end{lemma}

\begin{proof}
Properties (i), (ii) and (iii) follow from the orthogonal decomposition 
and the rotational invariance of the normal distribution. 

Property (iv) follows from a contraction property of conditional independence. 
Let us denote by $Y \perp B$ the independence of random variables (or vectors) $Y$ and $B$, 
and by $(Y \perp B) \mid H$ the conditional independence $Y$ and $B$ given $H$.
The contraction property states that 
$(Y \perp B) \mid H$ and $B \perp H$ imply $Y \perp B$.
In our situation, we have $(y_i \perp b_i) \mid \ip{a_i}{\bar{x}}$
by the assumption on $y_i$ and since $b_i$ is uniquely determined by $a_i$.
Moreover, $b_i \perp \ip{a_i}{\bar{x}}$ by property (iii).
The contraction property yields $y_i \perp b_i$. This proves (iv).
\end{proof}

\subsection{Proof of Proposition~\ref{prop: linear estimation}}

By the identical distribution of $a_i$ and using the 
orthogonal decomposition \eqref{eq: orthogonal decomposition}, 
we have
\begin{equation}         \label{eq: Exlin}
\E \xlin = \E y_1 a_1 = \E y_1 \ip{a_1}{\bar{x}} \bar{x} + \E y_1 b_1.
\end{equation}
The first term in the right hand side equals $\mu \bar{x}$ by definition of $\mu$.
The second term equals zero since by the independence property 
(iv) in Lemma~\ref{lem: orthogonal decomposition} and and since $\E b_1=0$.
We proved the first part of the proposition, $\E \xlin = \mu \bar{x}$.

To prove the second part, we express
$$
\E \|\xlin - \mu \bar{x}\|_2^2 = \E \Big\| \frac{1}{m} \sum_{i=1}^m Z_i \Big\|_2^2,
$$
where $Z_i = y_i a_i - \mu \bar{x}$ are independent and identically distributed
random vectors with zero mean. Thus
\begin{equation}         \label{eq: MSE through Z1}
\E \|\xlin - \mu \bar{x}\|_2^2 = \frac{1}{m^2} \sum_{i=1}^m \E \|Z_i\|_2^2 = \frac{1}{m} \E \|Z_1\|_2^2.
\end{equation}
Using orthogonal decomposition \eqref{eq: orthogonal decomposition} again, 
we can express $Z_1$ as follows:
$$
Z_1 = y_1 \ip{a_1}{\bar{x}} \bar{x} + y_1 b_1 - \mu \bar{x} = X+Y
$$
where 
$$
X = \big[y_1 \ip{a_1}{\bar{x}} - \mu \big] \bar{x}, \quad Y = y_1 b_1.
$$
Note that 
\begin{equation}         \label{eq: XY 0}
\quad \E \ip{X}{Y} = 0. 
\end{equation}
To see this, 
$\ip{X}{Y} = [y_1^2 \ip{a_1}{\bar{x}} - \mu y_1] \cdot \ip{b_1}{\bar{x}}$.
The two terms forming this product are independent by 
properties (iii) and (iv) in Lemma~\ref{lem: orthogonal decomposition}.
Moreover, $\E b_1 = 0$, which yields $\E \ip{b_1}{\bar{x}} = 0$,
and as a consequence, \eqref{eq: XY 0} follows.

Property \eqref{eq: XY 0} implies that   
$$
\E\|Z_1\|_2^2 = \E \|X+Y\|_2^2 = \E \|X\|_2^2 + \E \|Y\|_2^2.
$$
We have $\E\|X\|_2^2=\s^2$ by definition of $\s$ and since $\|\bar{x}\|_2 = 1$.
Next, $\E\|Y\|_2^2 = \eta^2 \E\|b_1\|_2^2$ by the independence property 
(iv) in Lemma~\ref{lem: orthogonal decomposition} 
and by definition of $\eta$.
Recalling that $b_1$ is a standard normal random variable in the hyperplane $x^\perp$, 
we get $\E\|b_1\|_2^2 = n-1$. It follows that
$$
\E\|Z_1\|_2^2 = \s^2 + \eta^2(n-1).
$$
Putting this into \eqref{eq: MSE through Z1}, we complete the proof.
\qed

\subsection{Metric projection and dual norm}

For a subset $T$ of $\R^n$, we define
\begin{equation}         \label{eq: dual norm}
\|x\|_{T^\circ} = \sup_{u \in T} \ip{x}{u}, \quad x \in \R^n.
\end{equation}
It is a semi-norm if $T$ is symmetric. We define also $T_d = (T-T) \cap d B_2^n$.

\begin{lemma}				\label{lem: projection}
  Let $K$ be an arbitrary subset of $\R^n$ and let $z \in K$, $w \in \R^n$.
  Then the distance
  $$
  d := \|P_K(w) - z\|_2
  $$
  satisfies the inequality
  $$
  d^2 \le 2 \|w-z\|_{K_d^\circ} 
  $$
\end{lemma}

\begin{proof}
By definition, $P_K(w)$ is the closest vector to $w$ in $K$, so 
$$
\|P_K(w)-w\|_2 \le \|z-w\|_2.
$$
We write this as 
$$
\|(P_K(w)-z) + (z-w)\|_2^2 \le \|z-w\|_2^2, 
$$
expand the left hand side and cancel the terms $\|z-w\|_2^2$ on both sides. 
This gives
\begin{equation}         \label{eq: PKw-z}
\|P_K(w)-z\|_2^2 \le 2 \ip{P_K(w)-z}{w-z}.
\end{equation}
The left hand side of \eqref{eq: PKw-z} equals $d^2$ by definition. 
To estimate the right hand side, note that vectors $P_K(w)$ and $z$ lie in $K$
and they are $d$ apart in Euclidean distance. Therefore 
$P_K(w)-z \in (K-K) \cap d B_2^n = K_d$. 
Thus the right hand side of \eqref{eq: PKw-z} is bounded by 
$2 \|w-z\|_{K_d^\circ}$ as claimed.
\end{proof}

\begin{corollary}					\label{cor: projection}
  Let $K$ be a star-shaped set and let $z \in K$, $w \in \R^n$.
  Then for every $t>0$ we have 
  \begin{equation}         \label{eq: projection}
  \|P_K(w)-z\|_2 \le \max \Big( t, \, \frac{2}{t} \|w-z\|_{K_t^\circ} \Big).
  \end{equation}
\end{corollary}

\begin{proof}
The proof relies on the following fact, which is well known in geometric functional analysis.

\begin{claim}
  For any fixed $x \in \R^n$, the function $\frac{1}{t} \|x\|_{K_t^\circ}$ is non-increasing in $t \in \R_+$.
\end{claim}

To prove this claim, we express the function as
\begin{equation}         \label{eq: 1twKt}
\frac{1}{t} \|x\|_{K_t^\circ} 
= \frac{1}{t} \E \sup_{u \in (K-K) \cap t B_2^n} \ip{x}{u}
= \E \sup_{v \in \frac{1}{t} (K-K) \cap B_2^n} \ip{x}{v}.
\end{equation}
Since $K$ is star-shaped, $K-K$ is star-shaped, too. Then the set $\frac{1}{t}(K-K)$
is non-increasing (with respect to inclusion) in $t \in \R_+$. 
Thus the same is true about $\frac{1}{t}(K-K) \cap B_2^n$. 
This and the identity \eqref{eq: 1twKt} prove the claim.

\medskip

To deduce \eqref{eq: projection}, denote $d = \|P_K(w)-z\|_2$.
If $d \le t$ then \eqref{eq: projection} holds. 
If $d>t$, we apply Lemma~\ref{lem: projection} followed by the claim above. 
We obtain 
$$
d \le \frac{2}{d} \|w-z\|_{K_d^\circ} \le \frac{2}{t} \|w-z\|_{K_t^\circ}.
$$
This implies \eqref{eq: projection}.
\end{proof}

\subsection{Proof of Theorem~\ref{thm: main}.}

\subsubsection{Decomposition of the error}

We apply Corollary~\ref{cor: projection} for $w = \xlin$ and $z = \mu \bar{x}$; 
note that the requirement that $z \in K$ is satisfied by assumption.
We obtain
\begin{equation}         \label{eq: error via dual norm}
\|\xhat - \mu \bar{x}\|_2 
= \|P_K(\xlin) - \mu \bar{x}\|_2
\le t + \frac{2}{t} \|\xlin - \mu \bar{x}\|_{K_t^\circ}.
\end{equation}

Recall that 
$\xlin = \frac{1}{m} \sum_{i=1}^m y_i a_i$ and use the orthogonal decomposition of $a_i$
from \eqref{eq: orthogonal decomposition}. By triangle inequality, we obtain 
\begin{align}         \label{eq: E1+E2}
\|\xlin - \mu \bar{x}\|_{K_t^\circ}
  &= \Big\| \frac{1}{m} \sum_{i=1}^m \big[ y_i \ip{a_i}{\bar{x}} \bar{x} + y_i b_i - \mu \bar{x} \big] \Big\|_{K_t^\circ} \nonumber\\
  &\le \Big\| \frac{1}{m} \sum_{i=1}^m \big[ y_i \ip{a_i}{\bar{x}} - \mu \big] \bar{x} \Big\|_{K_t^\circ}
    + \Big\| \frac{1}{m} \sum_{i=1}^m y_i b_i \Big\|_{K_t^\circ} 
  =: E_1 + E_2.
\end{align}

\subsubsection{Estimating $E_1$}

Denoting $\xi_i = y_i \ip{a_i}{\bar{x}} - \mu$, we have
$$
E_1 = \Big| \frac{1}{m} \sum_{i=1}^m \xi_i \Big| \cdot \|\bar{x}\|_{K_t^\circ}.
$$
By definition, $K_t \subseteq t B_2^n$ and $\bar{x} \in B_2^n$, so $\|\bar{x}\|_{K_t^\circ} \le t$.
Further, $\xi_i$ are independent and identically distributed random variables with zero mean. 
Therefore
\begin{equation}         \label{eq: E1}
\E E_1^2 \le \frac{1}{m} \E[\xi_1^2] \cdot t^2 = \frac{\s^2 t^2}{m},
\end{equation}
where the last equality follows by the definition of $\xi_1$ above and the definition of 
$\sigma$ in \eqref{eq: mu sigma eta}.

\subsubsection{Estimating $E_2$}

We will estimate
$$
\E E_2 = \E \|h\|_{K_t^\circ} 
\quad \text{where} \quad 
h := \frac{1}{m} \sum_{i=1}^m y_i b_i.
$$
Recall that $y_i$ and $b_i$ are independent by property (iv) in Lemma~\ref{lem: orthogonal decomposition}. 
Let us condition on $y_i$. 
Rotation invariance of the normal distribution implies that
$$
h \sim N(0, s I_{x^\perp})
\quad \text{where} \quad 
s^2 =\frac{1}{m^2} \sum_{i=1}^m y_i^2.
$$ 
In order to extend the support of the distribution of $h$ from $x^\perp$ to $\R^n$, 
we recall the following elementary and well known fact. 

\begin{claim}
  For a subspace $E$ of $\R^n$, let $g_E$ denote a random vector with distribution $N(0,I_E)$.
  Let $\Phi : \R^n \to \R$ be a convex function. Then 
  $$
  \E \Phi(g_E) \le \E \Phi(g_F) \quad \text{whenever} \quad E \subseteq F.
  $$
\end{claim}

To prove the claim, consider the orthogonal decomposition $F = E \oplus L$ for an appropriate
subspace $L$. Then $g_F = g_E + g_L$ in distribution, where $g_L \sim N(0,I_L)$ is independent of $g_E$.
Jensen's inequality yields
$$
\E \Phi(g_E) = \E \Phi(g_E + \E g_L) \le \E \Phi(g_E + g_L) = \E \Phi(g_F)
$$
as claimed. 

\medskip

In the notation of the claim, we can estimate the expectation of $\|h\|_{K_t^\circ}$ 
(still conditionally on the $y_i$'s) as follows:
$$
\E \|h\|_{K_t^\circ} 
= \E \|s g_{x^\perp}\|_{K_t^\circ}
= s \E \|g_{x^\perp}\|_{K_t^\circ}
\le s \E \|g_{\R^n}\|_{K_t^\circ}
= s \cdot w_t(K)
$$
where the last equality follows by the definitions of the dual norm and local mean width. 
Therefore, unconditionally we have 
\begin{align}				\label{eq: E2}
\E E_2 
  &= \E \|h\|_{K_t^\circ} 
  = \E[s] \cdot w_t(K)
  \le (\E s^2)^{1/2} \cdot w_t(K) \nonumber\\
  &= \big[ \frac{1}{m} \E y_1^2 \big]^{1/2} \cdot w_t(K)
  = \frac{\eta}{\sqrt{m}} \cdot w_t(K)
\end{align}
where the last equality follows from the definition of 
$\eta$ in \eqref{eq: mu sigma eta}.

It remains to plug estimates \eqref{eq: E1} for $E_1$ and \eqref{eq: E2} for $E_2$
into \eqref{eq: E1+E2}, and we obtain 
$$
\E \|\xlin - \mu \bar{x}\|_{K_t^\circ}
= \E(E_1+E_2)
\le (\E E_1^2)^{1/2} + \E E_2
\le \frac{1}{\sqrt{m}} \big[ \s t + \eta \, w_t(K) \big].
$$
Finally, substituting this estimate into \eqref{eq: error via dual norm} and taking 
expectation of both sides completes the proof of Theorem~\ref{thm: main}.
\qed

\section{High-probability version of Theorem~\ref{thm: main}}					\label{s: thm: main whp}

Using standard concentration inequalities, we can complement Theorem~\ref{thm: main} 
with a similar result that is valid with high probability rather than in expectation. 

To this end, we will assume that the observations $y_i$ have {\em sub-gaussian}
distributions. This is usually expressed by requiring that 
\begin{equation}         \label{eq: sub-gaussian}
\E \exp(y_i^2/\psi^2) \le 2 \quad \text{for some } \psi > 0.
\end{equation}
Basic facts about sub-gaussian random variables can be found e.g. in \cite{V}.  

To understand this assumption, consider the case when the non-linearity is given by an explicit function, i.e., $y_i = f(\langle a_i, x \rangle)$.  Since $\langle a_i, x \rangle$ is Gaussian, $y_i$ will be sub-Gaussian provided that $f$ does not grow faster than linearly, i.e., provided that $f(x) \leq a + b |x|$ for some scalars $a$ and $b$.

\begin{theorem}[Non-linear estimation with high probability]				\label{thm: main whp}
  Let $\bar{x} = x/\|x\|_2$.
  Assume that $\mu \bar{x} \in K$ where $K$ is a fixed star-shaped closed subset of $\R^n$.
  Assume that observations $y_i$ satisfy the sub-gaussian bound \eqref{eq: sub-gaussian}.
  Let $t > 0$ and $0 < s \le \sqrt{m}$.
  Then the non-linear estimator $\xhat$ defined in \eqref{eq: xhat} satisfies
  $$
  \|\xhat - \mu \bar{x}\|_2 \le t + \frac{4 \eta}{\sqrt{m}} \Big[ s + \frac{w_t(K)}{t} \Big]
  $$
  with probability at least $1 - 2\exp(-c s^2 \eta^4 / \psi^4)$. 
  Here $\mu$ and $\eta$ are the numbers defined in \eqref{eq: mu sigma eta},
  and $c>0$ is an absolute constant.
\end{theorem}

\begin{proof}
The conclusion follows by combining the proof of Theorem~\ref{thm: main} 
and standard concentration techniques which can be found in \cite{V}. 
So we will only outline the argument. Let $\e = s/(2\sqrt{m})$; then $\e<1/2$ by assumption.

First we bound $E_1$. Since $y_i$ are sub-gaussian as in \eqref{eq: sub-gaussian} 
and $\ip{a_i}{\bar{x}}$ is $N(0,1)$, the random variables $\xi_i$ are sub-exponential. More
precisely, we have $\E \exp(\xi_i / C\psi) \le 2$, where $C>0$ denotes an absolute constant here and thereafter.
A Bernstein-type inequality (see \cite[Corollary 5.17]{V}) implies that 
\begin{equation}         \label{eq: E1 whp}
\Big| \frac{1}{m} \sum_{i=1}^m \xi_i \Big| \le \e \eta, \quad \text{and thus} \quad E_1 \le \e \eta t,
\end{equation}
with probability at least 
\begin{equation}         \label{eq: prob E1}
1 - 2 \exp \Big[ -c \min \Big( \frac{\e^2 \eta^2}{\psi^2}, \, \frac{\e \eta}{\psi} \Big) m \Big]
\ge 1 - 2 \exp \Big( -\frac{c m \e^2 \eta^2}{\psi^2} \Big).
\end{equation}
In the last inequality we used that $\eta = \E y_i^2 \le C \psi$ by definition of $\psi$ 
and that $\e \le 1$ by assumption. 

Turning to $E_2$, we need to bound $\sum_{i=1}^m y_i^2$ and $\|g\|_{K_t^\circ}$.
A similar application of a Bernstein-like inequality for the sub-exponential random variables $y_i^2 - \eta^2$
shows that
$$
\sum_{i=1}^m y_i^2 \le 4 \eta^2 m
$$
with probability at least 
\begin{equation}         \label{eq: prob E2 y}
1 - 2 \exp \Big[ -c \min \Big( \frac{\eta^4}{\psi^4}, \, \frac{\eta^2}{\psi^2} \Big) m \Big]
\ge 1 - 2 \exp \Big(- \frac{c m \eta^4}{\psi^4} \Big).
\end{equation}
Next, we bound $\|g\|_{K_t^\circ}$ using Gaussian concentration. 
Since $K_t \subseteq t B_2^n$, the function $x \mapsto \|x\|_{K_t^\circ}$ on $\R^n$ 
has Lipschitz norm at most $t$. Therefore, the Gaussian concentration inequality 
(see e.g. \cite{Ledoux}) implies that 
$$
\|g\|_{K_t^\circ} \le \E \|g\|_{K_t^\circ} + t \e \sqrt{m} = w_t(K) + t \e \sqrt{m} 
$$
with probability at least 
\begin{equation}         \label{eq: prob E2 b}
1 - \exp(-c \e^2 m).
\end{equation}
If both $\sum_{i=1}^m y_i^2$ and $\|g\|_{K_t^\circ}$ are bounded as above, then 
$$
E_2 \le \frac{1}{m} \Big( \sum_{i=1}^m y_i^2 \Big)^{1/2} \, \|g\|_{K_t^\circ}
\le \frac{2 \eta}{\sqrt{m}} \big( w_t(K) + t \e \sqrt{m} \big).
$$

If also $E_1$ is bounded as in \eqref{eq: E1 whp}, then we conclude 
as in the proof of Theorem~\ref{thm: main} that
$$
\|\xhat - \mu \bar{x}\|_2 \le t + \frac{2}{t} (E_1 + E_2)
\le t + \frac{4 \eta}{\sqrt{m}} \cdot \frac{w_t(K)}{t} + 6 \eta \e.
$$
Recalling \eqref{eq: prob E1}, \eqref{eq: prob E2 y} and \eqref{eq: prob E2 b}, we see 
that this happens with probability at least
$$
1 - 2 \exp \Big( -\frac{c m \e^2 \eta^2}{\psi^2} \Big) 
  - 2 \exp \Big(- \frac{c m \eta^4}{\psi^4} \Big)
  - \exp(-c \e^2 m)
\ge 1 - 2 \exp \Big( - \frac{c' m \e^2 \eta^4}{\psi^4} \Big)
$$
for an appropriate absolute constant $c'>0$. (Here we used again that $\eta/\psi<1$ and $\e < 1$.)
The conclusion of Theorem~\ref{thm: main whp} follows by definition of $\e$. 
\end{proof}

\section{Proofs of Theorem \ref{thm:linear lower bound}, Corollary \ref{cor:linear lower bound}, and Lemma \ref{lem: conditions on f}.}
\label{sec: optimality proofs}

In this section, for a set $K \subset \R^n$ we use the notation $K_t = K \cap t B_2^n$. (Note that we do not symmetrize $K$ first.)

\subsection{Proof of Theorem \ref{thm:linear lower bound}}
The theorem is proven with a careful application of Fano's inequality, an information theoretic method which is useful for determining minimax estimation rates.  We state a version of Fano's inequality synthesized for lower-bounding expected error in the linear model (the steps needed to specialize Fano's inequality can be seen in  \cite{RWY_Minimax,CD_lower_bound}, for example).  For the general version see \cite{information_theory}.  
In all that follows, $\XX_t$ is a $t/10$-packing of $K_t$ with maximal cardinality.  Thus, $|\XX_t| = P_t$.

\begin{theorem}[Fano's Inequality for expected error]
\label{thm: fano}
Consider a fixed matrix $A$ in the linear model of Section~\ref{ssec: linear model}, and let $t > 0$.  Suppose that
\begin{equation}
\label{eq:Fano req}
\frac{1}{\nu^2 |\XX_t|^2} \sum_{\overset{w \neq v}{w,v \in \XX_t}} \twonorm{A (w - v)}^2 + \log(2) \leq \frac{1}{2} \log(|\XX_t|).
\end{equation}
Then there exists a constant $c > 0$ such that for any estimator $\xhat(A, y)$,
\[\sup_{x \in \XX_t} \E \twonorm{\xhat - x} \geq c t.\]
\end{theorem}

Theorem \ref{thm:linear lower bound} will be proven by conditioning on $A$ and applying Fano's inequality.  We begin to control the left-hand side of \eqref{eq:Fano req} by showing that $A$ does not significantly increase the average distance between points in our packing.  Indeed, 
\[\E \sum_{\overset{w \neq v}{w,v \in \XX_t}} \twonorm{A (w - v)}^2 \leq \|\E A^* A\|  \sum_{\overset{w \neq v}{w,v \in \XX_t}} \twonorm{(w - v)}^2 \leq m \sum_{\overset{w \neq v}{w,v \in \XX_t}} \twonorm{(w - v)}^2 \leq  4 m |\XX_t|^2 t^2. \]
The last step follows since $\XX_t \subset t B_2^n$.  Thus, by Markov's inequality, 
\begin{equation}
\label{eq: KL bound}
\sum_{\overset{w \neq v}{w,v \in \XX_t}} \twonorm{A (w - v)}^2 \leq 16 m |\XX_t|^2 t^2 \qquad \text{ with probability at least } 3/4.
\end{equation}
We pause to note that the above inequality is the only property of $A$ that is needed.

By conditioning on the good event that the Inequality \eqref{eq: KL bound} holds, the following Corollary gives a significant simplification of Theorem \ref{thm: fano}.  
 
\begin{corollary}[Simplified Fano Bound]
Consider the linear model of Section~\ref{ssec: linear model}, suppose $\|\E A A^*\| \leq m$, and let $t > 0$.  Set $\bar{\nu} := \nu/\sqrt{m}$.  Suppose that
\begin{equation}
\label{eq: reduced Fano}
\frac{t^2}{\bar{\nu}^2} + 1 \leq c \log(|\XX_t|)
\end{equation}
for a small enough constant $c$.
Then there exists a constant $c' > 0$ such that for any estimator $\xhat(A, y)$,
\[\sup_{x \in \XX_t} \E \twonorm{\xhat - x} \geq c' t.\]
\end{corollary}

We are now in position to prove Theorem \ref{thm:linear lower bound}.  
\begin{proof}

Let $c_0$ be a small numerical constant, whose value will be chosen at the end of the proof, and let
\begin{equation}
\label{eq: define delta}
\delta := c_0 \delta_* = c_0 \inf_t \big\{ t + \bar{\nu}\Big[1 + \sqrt{\log P_t}\Big]\big\}.
\end{equation}

We split the proof into two cases, depending on the relative size of $\delta$ and $\bar{\nu}$.

\noindent{\bf Case 1: }  $\delta \geq \bar{\nu}$.

This case contains the meat of the proof, but nevertheless it can be proven quite quickly with the tools we have gathered.  Since $\delta$ is the infimum in \eqref{eq: define delta}, it satisfies
\[\delta \leq c_0 \left(\delta + \bar{\nu} \Big[1 + \sqrt{\log P_\delta}\Big]\right).\]
Further, $\bar{\nu} \leq \delta$ in Case 1.  Thus, we may massage the equation to give
\[\delta \leq c_1 \bar{\nu}\sqrt{\log P_\delta}\]
where $c_1 = c_0/(1 - 2 c_0)$ and we take $c_0 < 1/2$.

We now check that $\delta$ satisfies Inequality \eqref{eq: reduced Fano} as required for Fano's inequality.  Since $\d \geq \bar{\nu}$, one has
\[\frac{\d^2}{\bar{\nu}^2} + 1 \leq 2 \frac{\d^2}{\bar{\nu}^2} \leq 2 c_1^2 \log P_\delta.\]
The inequality now follows if we take $c_0$ such that $2 c_1^2 < c$.

We now move to the second case.

\noindent{\bf Case 2:} $\delta \le \bar{\nu}$.
This is an unusual case; it only occurs when $K$ is quite small, e.g., a 1-dimensional subspace.
Fano's inequality, which is more applicable for higher dimensional sets, may not give the optimal answer, but we may use much simpler methods.  We give a bare-bones argument. 

Consider any two signals, $w, v \in K$ satisfying $\twonorm{w - v} \leq c \bar{\nu}$.  We will show that these two points are indistinguishable based on the noisy observations $y$.  Consider the hypothesis testing problem of determining which of the two points generate $y$, i.e., either $y = A w + \e$ or $y = A v + \e$.  As seen from \eqref{eq: KL bound}, with probability at least $3/4$, $\twonorm{A(v - w)} \leq c \sqrt{m} \bar{\nu} \leq c \nu$.  Thus, the difference between the two candidate signals is much smaller than the noise level.  It follows from a quick analysis of the \textit{likelihood ratio test}, which is optimal by \textit{Neyman-Pearson lemma}, that the  that hypotheses are indistinguishable, i.e., there is no estimator which has a more than $3/4$ chance of determining the original signal.  This gives the lower bound
\[\sup_{x \in K} \E \twonorm{\xhat - x} \geq c \twonorm{v - w}.\]
By taking the maximum distance between signals $v, w \in K \cap \bar{\nu} B_2^n$, this gives
\[\sup_{x \in K} \E \twonorm{\xhat - x} \geq c\min(\diam(K), \bar{\nu}) \geq c \min(\diam(K), \delta).\]
\end{proof}

\subsection{Proof of Corollary \ref{cor:linear lower bound} incorporating Remark \ref{rem: alpha at scale}}

We prove a generalization of Corollary \ref{cor:linear lower bound} that takes the tighter version of $\alpha$ from Remark \ref{rem: alpha at scale}.

\begin{corollary}
Assume that $x \in K$ where $K$ is a star-shaped subset of $\R^n$.  Assume that the observations $y_i$ are given by \ref{eq: linear model Gaussian noise}.  Let 
\begin{equation}
\delta^* :=  \inf_t \big\{ t + \frac{\nu}{\sqrt{m}}\Big[1 + \frac{w_t(K)}{t}\Big]\big\}.
\end{equation}
Then any estimator $\xhat$ satisfies
\[
\sup_{x \in K} \E \twonorm{\xhat - x} \geq c_\alpha \min(\delta^*, \diam(K)),
\]
where $c_\alpha = c/\alpha$ for a numerical constant $c$ and $\alpha$ defined in \eqref{eq: alpha at scale}.
\end{corollary}

\begin{proof}
The proof proceeds simply by showing that $\delta^*_2 = \frac{1}{2} \delta^* \leq \alpha \delta_*$, with $\delta_*$ defined in \eqref{eq: def delta}.  One may then apply Theorem \ref{thm:linear lower bound}.  First, observe that by definition
\[\delta^*_2 \leq \frac{1}{2}\left(\delta^*_2 + \frac{\nu}{\sqrt{m}}\Big[1 + \frac{w_{\delta^*_2}(K)}{\delta^*_2}\Big]\right).\]
Massage the equation to give
\[\delta^*_2 \leq \frac{\nu}{\sqrt{m}}\Big[1 + \frac{w_{\delta^*_2}(K)}{\delta^*_2}\Big].\]
The Geometric Assumption then implies that 
\[\delta^*_2 \leq \frac{\nu}{\sqrt{m}}\Big[1 + \alpha \sqrt{\log P_{\delta^*_2}}\Big] \leq \alpha \cdot \frac{\nu}{\sqrt{m}}\Big[1 + \sqrt{\log P_{\delta^*_2}}\Big].\]
Note that since $K$ is star-shaped, $P_t$ is decreasing in $t$ and thus,
\[\delta^*_2/\alpha \leq \inf_{t \leq \delta^*_2} \frac{\nu}{\sqrt{m}}\Big[1 + \sqrt{\log P_{\delta^*_2}}\Big] \leq \inf_{t \leq \delta^*_2} \{t + \frac{\nu}{\sqrt{m}}\Big[1 + \sqrt{\log P_{\delta^*_2}}\Big]\}.\]
Now, if $\delta^*_2 < \delta_*$, then the proof is done, so we may restrict to the case $\delta^*_2 \geq \delta_*$.  However, in this case, by definition of $\delta_*$, the restriction $t \leq \delta^*_2$ may be removed from the infimum just above, and so
\[\delta^*_2/\alpha \leq \inf_t \{t + \frac{\nu}{\sqrt{m}}\Big[1 + \sqrt{\log P_{\delta^*_2}}\Big]\} = \delta_*.\]
\end{proof}

\subsection{Proof of Lemma \ref{lem: conditions on f}}
\begin{proof}[Proof of Lemma \ref{lem: conditions on f}]  
The proof will follow from several manipulations and decompositions of Gaussian random variables.  By definition, in the nonlinear model
\[\mu = \E f(\ip{a_1}{x} + \e_1) \ip{a_1}{\bar{x}} = \E f(\twonorm{x} \cdot g + \nu z) g\]
where $g, z \sim N(0,1)$ are independent.  Let $\beta^2 = \Var(\twonorm{x} \cdot g + \nu z) = \twonorm{x}^2 + \eta^2$ and let
\[w := \frac{\twonorm{x} \cdot g + \nu z}{\beta} \sim N(0,1).\]
Further, decompose $g$ as $g = (\twonorm{x}/\beta)  w + w_\perp$ where $w_\perp$ is independent of $w$.  Thus, 
\[\mu \beta/\twonorm{x} = \E f(\beta w) w = \E \left[f(\beta \abs{w}) \cdot \abs{w}\right].\]
The second equality follows since $f$ is assumed to be odd.  Now, since $f$ is nondecreasing, we have 
\[
\E f(\beta \abs{w}) \cdot \abs{w} \geq f(\beta) \cdot \Pr{\abs{w} \geq 1} = C f(\beta),
\]
where $C = \Pr{\abs{w} \geq 1} \approx 0.683$.

Putting pieces together, we have
\[\mu \geq C\frac{\twonorm{x} f(\beta)}{\beta}.\]
We now give an upper bound for $\sigma$ and $\eta$. 
\[\sigma^2 = \Var(f(\beta w) \cdot g) \leq \E (f(\beta w) \cdot g)^2.\]
By Cauchy-Schwarz inequality, the right-hand side is bounded by
\begin{equation}
\label{eq: bound sigma}
 \sqrt{\E f^4(\beta w)} \cdot \sqrt{ \E g^4} = 3^{1/2} \sqrt{\E f^4(\beta w)}.
 \end{equation}
The sub-multiplicative assumption implies that
\[\E f^4(\beta w) \leq f^4(\beta) \E f^4 (w) = C_f \cdot f^4(\beta).\]
Thus, $\sigma \leq C_f f(\beta)$.  As a bonus, we bounded $\eta$ by the same quantity, since $\eta^2$ is bounded by the right-hand side of \eqref{eq: bound sigma}.  Thus, 
\[ \lambda (\sigma + \eta) = \frac{\twonorm{x}}{\mu} (\sigma + \eta) \leq \frac{\twonorm{x} f(\beta) \beta}{\twonorm{x}} C_f f(\beta) \leq C_f (\twonorm{x} + \eta),\]
where
\[
C_f = C \E[f^4(w)]^{1/4}, \quad C = 48^{1/4}\Pr{\abs{w}\geq 1} \approx 1.8.
\]

\end{proof}

\bibliographystyle{abbrv}
\bibliography{pvy-high-dim-estimation-geometric}

\end{document}